\documentclass[11pt,twoside]{amsart}
\usepackage{hyperref, tikz, fullpage}
\usepackage{placeins, enumerate,longtable}
\usepackage{amsthm, amssymb,amsmath,latexsym}
\usepackage[utf8]{inputenc}
\usepackage[T1]{fontenc}
\usepackage{enumitem,color}
        
\usepackage{caption}
\usepackage{lscape}
\newtheorem{theorem}{Theorem}[section]
\newtheorem{prop}[theorem]{Proposition}
\newtheorem{lemma}[theorem]{Lemma}
\newtheorem{corollary}[theorem]{Corollary}
\newtheorem{defn}[theorem]{Definition}

\numberwithin{equation}{section}
\newcommand{\Cent}{\mathrm{Central}}
\newcommand{\Reg}{\mathrm{Regular}}

\newcommand{\Irr}{\mathrm{Irr}}

\newcommand{\F}{\mathbb{F}}
\newcommand{\Fq}{\mathbf{F}_q}
\newcommand{\Fqq}{\mathbf{F}_{q^2}}

\newcommand{\bmf}[1]{\mathbf{#1}}

\newcommand{\lbd}{\lambda}

\newcommand{\sbteq}{\subseteq}

\newcommand{\Zu}[1]{Z_{U_{#1}(\Fq)}}
\newcommand{\U}[1]{U_{#1}(\Fq)}

\begin{document}
\title{Branching rules for Unitary and Symplectic matrices}
\author{Uday Bhaskar Sharma}
\email{udaybsharmaster@gmail.com}
\author{Anupam Singh}
\email{anupamk18@gmail.com}
\address{Indian Institute of Science Education and Research (IISER) Pune,  Dr. Homi Bhabha Road Pashan, Pune 411008, India}
\today
\thanks{The first named author has received National Post-Doctoral Fellowship from SERB, India during this project. The second named author would like to acknowledge support of MATRICS grant from SERB during this project.}

\begin{abstract}
This paper concerns the enumeration of simultaneous conjugacy classes of tuples of commuting unitary matrices and of commuting symplectic matrices over a finite field $\Fq$ of odd size. For any given conjugacy class, the orbits for the action of its centralizer group on itself by conjugation (that is, the conjugacy classes within the centralizer group) are called branches. We determine the branching rules for the unitary groups $U_2(\Fq), U_3(\Fq)$, and for the symplectic groups $Sp_2(\Fq), Sp_4(\Fq)$. 
\end{abstract}
\maketitle
\section{Introduction}
Let $G$ be a finite group and $k$ a positive integer. The group $G$ acts on the set, 
$$G^{(k)} = \{(g_1,\ldots, g_k)\in G^k \mid  g_ig_j = g_jg_i, \forall~ 1\leq i \neq j \leq k \}$$
of $k$-tuples of pair-wise commuting elements in $G$ by simultaneous conjugation. The question here is to determine the (number of) orbits, which is simultaneous conjugacy classes of commuting elements for a given $G$. The evaluation of the number of simultaneous conjugacy classes of commuting elements is done with the help of \emph{branching rules}. For a given conjugacy class $C(g)$ of $g$ in $G$, the orbits for the conjugacy action of the centralizer $Z_G(g)$ on itself (that is the conjugacy classes within $Z_G(g)$) are called \emph{branches} of the given conjugacy class $C(g)$. Notice that the computation of simultaneous conjugacy classes of commuting elements amounts to 
\begin{enumerate}
\item[(a)] computing branches of all conjugacy classes of $G$, 
\item[(b)] computing branches of all conjugacy classes of centralizer subgroups of $G$, and
\item[(c)] continue the step (b) for further centralizer subgroups within those centralizer subgroups iteratively. 
\end{enumerate}
For the step (a) we simply need to compute the branches for the types (the elements of which centralizer subgroups are isomorphic).  Now, the conjugacy classes in a particular $Z_G(g)$ are either a conjugacy class of $G$ or a new type (because of splitting of a class of $G$). For the steps (b) and (c) we compute the branches for the new types. This information is stored in a \emph{branching table} which is indexed by the types of conjugacy classes of $G$ alongwith the type of new ones. The $rs^{th}$ entry of the branching tables stores the number of times the class $C_s$ appears in branching of class $C_r$, thus giving rise to the branching matrix $B_G$. This paper concerns determination of this branching matrix for some matrix groups. The first named author studied this problem for the algebra of all matrices over finite field in~\cite{Sh1, Sh2} which also gives the answer for the group $GL_n(\Fq)$. In this paper we continue this study for unitary and symplectic groups.

One of the central problem in group theory is to understand commuting probability for a group. The probability of finding a $k$-tuple (or $k$ elements in $G$ randomly) of elements of $G$ which commute pair-wise is defined by 
$$cp_k(G) = \displaystyle\frac{|G^{(k)}|}{|G|^k}.$$
Erdos and Turan~\cite{ET} showed that $cp_2(G) = \displaystyle\frac{l}{|G|}$, where $l$ is the number of conjugacy classes in $G$. Lescot (see~\cite[Lemma 4.1]{Le}) gave a recurrence formula for $cp_k(G)$: 
\begin{equation}\label{ELesc}
cp_k(G) = \frac{1}{|G|} \sum_{i = 1}^{l}  \frac{cp_{k-1}\left (Z_G(g_i)\right)}{|{\text C}(g_i)|^{k-2}},
\end{equation} 
where ${\text C}(g_i)$ denotes the conjugacy class of $g_i$ in $G$, and $Z_G(g_i)$ denotes the centralizer of $g_i$ in $G$. Let $c_G(k)$ be the number of simultaneous conjugacy classes of $k$-tuples of pairwise commuting elements of $G$. We set $c_G(0)=1$. For the case $k\geq 3$, we use Lescot's formula to relate $cp_k(G)$ to $c_G(k-1)$ and the branching matrix $B_G$ as follows. 
\begin{theorem}\label{TMain}
Let $G$ be a finite group and $k$ be a positive integer. 
The probability that a $k$-tuple of elements of $G$ commute is 
$$ cp_k(G) = \frac{c_G(k-1)}{|G|^{k-1}} = \frac{{\bf 1} . B_G^{k-1} . e_1}{|G|^{k-1}}$$
where ${\bf 1}$ is a row matrix with all $1$'s, the vector $e_1$ is the column vector with $1$ at first place and $0$ elsewhere and $B_G$ is the branching matrix of $G$.
\end{theorem}
\noindent We prove this in Section~\ref{brcp}. Using this one can compute $cp_k(G)$ easily for those groups where we compute the branching matrix. We plan to continue our study along these lines too in future. For more information on the importance of commuting probability we refer to~\cite{FF,FG,GR,HR,Le}.

In~\cite{BMRT} the simultaneous conjugacy classes, the orbits of the action of $G$ on $G^k$ by conjugation, are studied in the context of complete reducibility for algebraic groups. The moduli spaces of ordered pairs and ordered triples of commuting elements in a compact Lie group, up to simultaneous conjugacy, is studied in~\cite{BFM}. In~\cite{PM} the indecomposable symplectic modules of $C_2\times C_2$ in characteristic $2$ are studied, which comes down to essentially looking at this problem for ordered pairs.  In general, the branching helps in computing the indecomposable (unitary, symplectic, orthogonal) $kA$-module where $A$ is a finite Abelian group. We compute the number of simultaneous conjugacy classes of $k$-tuples of commuting matrices in $U_n(\Fq)$ and $Sp_{2l}(\Fq)$ for $n=2,3$ and $l=1,2$, which gives the number of isomorphism classes of $n$-dimensional unitary representations and symplectic representations, of the polynomial algebra in $k$ variables $\Fqq[x_1, \ldots, x_k]$ and $\Fq[x_1, \ldots, x_k]$ respectively.

Let $\Fq$ be a finite field, where $q$ is a prime power, and $q$ is odd. Consider $\Fqq$, the degree $2$ field extension of $\Fq$ and the involution $\sigma\colon \Fqq \rightarrow \Fq$ given by $\sigma(x) = x^q$ for $x \in \Fqq$. Let $V$ be an $n$-dimensional vector space over $\Fqq$, and let $\beta\colon V\times V \rightarrow \Fqq$ be a non-degenerate sesquilinear form. We say that $\beta$ is hermitian if with respect to some basis for $V$, the matrix, $[\beta_{ij}]$ of $\beta$ satisfies $ {}^t[\beta_{ij}^q] = [\beta_{ij}]$. Let $g \in GL_n(\Fqq)$, where $g = [g_{ij}]$, and $\overline{g}= [g_{ij}^q]$. We say that $g$ is a \emph{unitary matrix} with respect to the hermitian form $\beta$ if $ {}^t\!g\beta\overline{g} = \beta$. The set $U_{n,\beta}(\Fqq) = \{g \in GL_n(\Fqq) \mid  {}^t\!g\beta \overline{g} = \beta\}$ forms a group, called a unitary group. All hermitian forms over a finite field are equivalent (see~\cite[Corollary 10.4]{Gro}), the unitary groups therefore are conjugate in $GL_n(\Fqq)$ for different $\beta$ and hence isomorphic. Thus, the unitary group is simply denoted by $U_n(\Fq)$. The relation between conjugacy classes and characters of this group has been a topic of intensive investigation. Wall~\cite{Wa} discussed the conjugacy classes of unitary, symplectic and orthogonal groups. Thiem and Vinroot~\cite{TV} studied the character theory of unitary groups by explaining the Ennola duality, which tries to generalize Green's results for the general linear group. The centralizer classes (z-classes or similarity class types) in $U_n(\Fq)$ is studied in~\cite{BS} and proved that this number is same as that of $GL_n(\Fq)$ when $q>n$. We compute the branching rules for the group $U_n(\Fq)$ when $n=2$ and $3$. 
\begin{theorem}\label{main1}
The branching rules for the unitary groups $U_2(\Fq)$ and $U_3(\Fq)$ are as in the branching table given in Table~\ref{tableI} and Table~\ref{tableII}, respectively.
\end{theorem}
\noindent The proof of this theorem is in two parts. The case of $U_2(\Fq)$ is in Section~\ref{U2Main1}, and the case of $U_3(\Fq)$ is dealt in Section~\ref{U3Main1}.

Now, consider finite field $\Fq$ of odd size and a vector space $V$ of dimension $2l$ over it. We consider the non-degenerate skew-symmetric bilinear form $\beta\colon V\times V \rightarrow \Fq$, which is unique up to equivalence, i.e., the matrix of $\beta$ satisfies $\beta = -{}^t\!\beta$. Then any element $g\in GL_{2l}(\Fq)$ is said to be \emph{symplectic} if $ g\beta\ \!{}^t\!g = \beta$. The set $Sp_{2l}(\Fq) = \{g \in GL_{2l}(\Fq)\mid g\beta {}^t\!g = \beta\}$ of symplectic matrices over $\Fq$, forms a group and is called the \emph{symplectic group}. The representation theory of symplectic group is studied in~\cite{Sr, Vi}. We compute the branching rules in this group.
\begin{theorem}\label{Main2}
The branching rules for the symplectic groups, $Sp_2(\Fq)\cong SL_2(\Fq)$ and $Sp_4(\Fq)$ are as in the branching table in Table~\ref{tableIII} and Table~\ref{tableIV} respectively.
\end{theorem}

\noindent The proof of this theorem too is in two parts. The proof in the case of $Sp_2(\Fq)$ is in Section~\ref{SecSp2}, and the proof in the case of $Sp_4(\Fq)$ is in Section~\ref{SecSp4}. We also see that for the case of $Sp_4(\Fq)$ three new branches appear and hence the table is of size $21+3=24$ instead of $21$ which is the number of conjugacy class types. In all other cases the branching table is of the same size as the number of conjugacy class types. 

In addition to determining the branching rules, we observe the following:
\begin{corollary}\label{CorMain1}
\begin{enumerate}
\item Let $c_u(2,k,q)$ be the number of simultaneous conjugacy classes of $k$-tuples of commuting matrices in $U_2(\Fq)$, and $c_g(2,k,q)$ the same for $GL_2(\Fq)$. The, $c_u(2,k,q)$ can be obtained from $c_g(2,k,q)$ by swapping the $(q-1)$'s with $(q+1)$'s in the formula of $c_g(2,k,q)$ for any $k$.
 \item Using SageMath~\cite{SAGE}, we show that for $k\leq 20$ the number of simultaneous conjugacy classes of commuting $k$-tuples in $U_2(\Fq)$, $U_3(\Fq)$, and $Sp_2(\Fq)$, are  polynomials in $q$ with non-negative integer coefficients.
\end{enumerate}
\end{corollary}
\noindent It would be interesting to prove the second result in this corollary, as the computations suggest, in its full generality. We hope to explore this in future.

\begin{table}
\caption{Branching table of $U_2(\Fq)$}
\label{tableI}
\begin{center}
\begin{tabular}{|c|c|c|c|}
\hline
$(1,1)_1$ & $(2)_1$ & $(1)_1(1)_1$ & $(1)_2$\\ \hline 
$q+1$ & $0$ & $0$ & $0$\\
$q+1$ & $q(q+1)$ & $0$ & $0$ \\
${q+1}\choose 2$ & $0$ & $(q+1)^2$ & $0$ \\
$\frac{q^2-q-2}{2}$ & $0$ & $0$ & $q^2 - 1$ \\\hline 
\end{tabular}
\end{center}
\end{table}

\begin{table}
\caption{Branching table of $U_3(\Fq)$}
\label{tableII}
\begin{center}
\SMALL  
\begin{tabular}{|c|c|c|c|c|c|c|c|}
\hline
$(1,1,1)_1$ & $(2,1)_1$ & $(1,1)_1(1)_1$ & $(3)_1$ & $(2)_1(1)_1$ & $ (1)_1(1)_1(1)_1$ & $(1)_2(1)_1$ & $(1)_3$ \\ \hline
$q+1$ & $0$ & $0$ & $0$ & $0$ & $0$ & $0$ & $0$ \\
$q+1$ & $q(q+1)$ & $0$ & $0$ & $0$ & $0$ & $0$ & $0$ \\
$q(q+1)$ & $0$ & $(q+1)^2$ & $0$ & $0$ & $0$ & $0$ & $0$\\
$q+1$ & $q^2 - 1$ & $0$ & $(q+1)q^2$ & $0$ & $0$ & $0$ & $0$ \\
$q(q+1)$ & $(q+1)q^2$ & $(q+1)^2$ & $0$ & $q^2(q+1)$ & $0$ & $0$ & $0$ \\
${q+1}\choose 3$ & $0$ &$(q+1){{q+1}\choose 2}$ & $0$ & $0$ & $(q+1)^3$ & $0$ & $0$ \\
$\frac{(q+1)(q^2-q-2)}{2}$ & $0$ & $\frac{(q+1)(q^2-q-2)}{2}$ & $0$ & $0$ & $0$ & $(q+1)(q^2-1)$& $0$\\
$\frac{q^3-q}{3}$ & $0$ & $0$ & $0$& $0$& $0$& $0$ &  $q^3 + 1$ \\\hline 
\end{tabular}
\end{center}
\end{table}

\begin{table}
\caption{Branching table of $Sp_2(\Fq)$}
\label{tableIII}
\begin{center}
\begin{tabular}{|c|c|c|c|c|}
\hline
$C$& $A_1$ & $A_2$ & $D$ & $Ir$ \\ \hline
$2$ & $0$ & $0$ & $0$ & $0$ \\
$2$ & $2q$ & $0$ & $0$ & $0$ \\
$2$ & $0$ & $2q$ & $0$ & $0$ \\
$\frac{q-3}{2}$ & $0$ & $0$ & $q - 1$ & $0$ \\
$\frac{q-1}{2}$ & $0$ & $0$ & $0$ & $q + 1$ \\ \hline 
\end{tabular}
\end{center}
\end{table}

\begin{landscape}
\begin{table}
\caption{Branching table of $Sp_4(\Fq)$}
\label{tableIV}
\begin{center}
$$    \left(\begin{smallmatrix}A_1&A_2&A_3&A'_3&A_4&B_1&B_2&B_3&B_4&B_5&B_6&B_7&B_8&B_9&C_1&C_2&C_3&C_4&D_1&D_2&D_3&N_1&N_2&N_3\\ \hline
&&&&&&&&&&&&&&&&&&&&&&& \\ 
     2&0&0&0&0&0&0&0&0&0&0&0&0&0&0&0&0&0&0&0&0&0&0&0\\
     4&2q&0&0&0&0&0&0&0&0&0&0&0&0&0&0&0&0&0&0&0&0&0&0\\
     2&0&2q&0&0&0&0&0&0&0&0&0&0&0&0&0&0&0&0&0&0&0&0&0\\
     2&0&0&2q&0&0&0&0&0&0&0&0&0&0&0&0&0&0&0&0&0&0&0&0\\
     4&4q-4&0&0&2q^2&0&0&0&0&0&0&0&0&0&0&0&0&0&0&0&0&0&0&0\\
     \frac{q^2-1}{4}&0&0&0&0&q^2+1&0&0&0&0&0&0&0&0&0&0&0&0&0&0&0&0&0&0\\
     \frac{(q-1)^2}{4}&0&0&0&0&0&q^2-1&0&0&0&q(q-2)&0&\frac{q(q-1)}{2}&0&0&0&0&0&0&0&0&0&0&0\\
     \frac{q^2-8q+15}{8}&0&0&0&0&0&0&(q-1)^2&0&0&0&0&\frac{q^2-3q+2}{2}&0&0&0&\frac{q^2-4q+3}{2}&0&\frac{(q-3)^2}{2}&0&0&0&0&0\\
     \frac{q^2-4q + 3}{8}&0&0&0&0&0&0&0&(q+1)^2&0&q+1&0&0&0&\frac{q^2-1}{2}&0&0&0&\frac{(q-1)^2}{4}&0&0&0&0&0\\
     \frac{q^2-4q+3}{4}&0&0&0&0&0&0&0&0&q^2-1&0&0&0&0&\frac{q^2-2q-3}{2}&0&\frac{(q-1)^2}{2}&0&\frac{q^2-4q+3}{2}&0&0&0&0&0\\
     \frac{q-1}{2}&0&0&0&0&0&0&0&0&0&q+1&0&0&0&0&0&0&0&0&0&0&0&0&0\\
     \frac{q-1}{2}&0&0&\frac{q^2-q}{2}&0&0&0&0&0&0&q+1&q^2+q&0&0&0&0&0&0&0&0&0&0&0&0\\
     \frac{q-3}{2}&0&0&0&0&0&0&0&0&0&0&0&q-1&0&0&0&0&0&0&0&0&0&0&0\\
     \frac{q-3}{2}&0&\frac{q^2-3q}{2}&0&0&0&0&0&0&0&0&0&q-1&q^2-q&0&0&0&0&0&0&0&0&0&0\\
     q-1&0&0&0&0&0&0&0&0&0&0&0&0&0&2q+2&0&0&0&2q-2&0&0&0&0&0\\
     2q-2&q^2-q&0&0&0&0&0&0&0&0&0&0&0&0&4q+4&2q^2 + 2q&0&0&4q-4&q^2-q&0&0&0&0\\
     q-3&0&0&0&0&0&0&0&0&0&0&0&0&0&0&0&2(q-1)&0&2q-6&0&0&0&0&0\\
     2q-6&q^2-3q&0&0&0&0&0&0&0&0&0&0&0&0&0&0&4q-4&2q^2-2q&4q-12&q^2-3q&0&0&0&0\\
     1&0&0&0&0&0&0&0&0&0&0&0&0&0&0&0&0&0&4&0&0&0&0&0\\
     4&2q&0&0&0&0&0&0&0&0&0&0&0&0&0&0&0&0&16&4q&0&0&0&0\\
     4&4q&2q^2&2q^2&0&0&0&0&0&0&0&0&0&0&0&0&0&0&16&8q&4q^2&2q^2&0&0\\
     0&4q&2q^2-2q&2q^2-2q&0&0&0&0&0&0&0&0&0&0&0&0&0&0&0&0&0&2q^2&0&0\\
     0&4&0&0&0&0&0&0&0&0&0&0&0&0&0&0&0&0&0&0&0&0&2q^3&0\\
     0&0&q^2+3q&q^2-q&0&0&0&0&0&0&0&0&0&0&0&0&0&0&0&0&0&q^2-q&0&2q^3
    \end{smallmatrix}\right)
$$    \end{center}
\end{table}
\end{landscape}

\subsection{Notations Used}
For any $k$-tuple $(A_1, A_2, \ldots, A_k)$ of pair-wise commuting matrices in $U_{n}(\Fq)$, the common centralizer of the tuple is the intersection of the centralizers of the $A_i$'s i.e., $\bigcap_{i = 1}^k \Zu{n}(A_i)$. We denote this by $\Zu{n}(A_1,A_2, \ldots, A_k)$.
The number of simultaneous conjugacy classes of $k$-tuples of commuting unitary $n \times n$ matrices with entries in $\Fqq$ is denoted as $c_u(n,k,q)$. For any $k$-tuple $(A_1, A_2, \ldots, A_k)$ of pair-wise commuting matrices in $Sp_{2l}(\Fq)$, the common centralizer $\bigcap_{i=1}^k Z_{Sp_{2l}(\Fq)}(A_i)$ is denoted by $Z_{Sp_{2l}}(\Fq)(A_1,\ldots, A_k)$.
The number of simultaneous conjugacy classes of $k$-tuples of commuting $2l\times 2l$ symplectic matrices over $\Fq$ is denoted as $c_s(2l,k,q)$.

\subsection*{Acknowledgments} The authors would like to thank Amritanshu Prasad, IMSc Chennai, for his interest in this work. We thank the referee(s) for their feedback and suggestions for revision which helped the paper improve. 

\section{Preliminaries for the unitary groups}\label{SecPrelim}
We need to understand the conjugacy classes in unitary groups. We briefly recall that here.
\subsection{Self $U$-reciprocal and $U$-irreducible Polynomials}
In this section, we shall discuss polynomials over $\Fqq$, called \emph{self $U$-reciprocal} polynomials. Let $f(t) \in \Fqq[t]$, be a monic polynomial $f(t) = t^d + a_{d-1}t^{d-1} + \cdots + a_1t + a_0$. 
\begin{defn}
Let $f(t) \in \Fqq[t]$ be a monic polynomial as above. Let $\tilde{f}(x)$ be the polynomial $\tilde{f}(t) = f(0)^{-q}t^d\left (t^{-d} + a_{d-1}^qt^{-d+1} + \cdots + a_1^qt^{-1} + f(0)^q \right)$. The polynomial $f(t)$ is said to be \emph{self $U$-reciprocal} if $\tilde{f}(t) = f(t)$. 
\end{defn}
\noindent In~\cite{BS} these polynomials are used to understand conjugacy of centralizer subgroups in unitary groups. We shall also define \emph{$U$-irreducible} polynomials.
\begin{defn}
We say that $f(t)\in \Fqq[t]$ is \emph{$U$-irreducible} if there is an $x \in \overline{\bmf{F}}$, the algebraic closure of $\Fq$, such that $f(t) = (t - x)(t - x^{-q})(t - x^{(-q)^2})\cdots (t - x^{(-q)^{d-1}})$. 
\end{defn}
\noindent Thiem and Vinroot in~\cite{TV} used these polynomials and called it \emph{$F$-irreducible}. It is a routine exercise to check that any $U$-irreducible polynomial in $\Fqq[t]$ is a self $U$-reciprocal polynomial. We now have a useful lemma due to Ennola (see~\cite[Lemma 2]{En}), which will help us determine when a $U$-irreducible polynomial is irreducible in the usual sense. For a proof see Section 2.1 in~\cite{TV}.
\begin{lemma}
A polynomial $f(t) \in \Fqq[t]$ is $U$-irreducible if and only if either
\begin{enumerate}
\item $f(t)$ is irreducible in $\Fqq[t]$, and $f$ is self $U$-reciprocal ($\deg(f(t))$ must be odd here), or,
\item $f(t) = h(t)\tilde{h}(t)$, where $h(t)$ is irreducible in $\Fqq[t]$, but is not self $U$-reciprocal ($\deg(f(t))$) is even in this case). 
\end{enumerate}
\end{lemma}

Next, we need to find out the number of $U$-irreducible polynomials of any given degree. Consider the polynomial $a_n(t) = t^{q^n - (-1)^n} - 1$. Any $x$ such that $x^{(-q)^d} = x$, for $d$ such that $d \mid n$, is a root of $a_n(t)$. Thus, every root of a $U$-irreducible polynomial whose degree divides $n$, is a root of $a_n(t)$. Hence, $a_n(t)$ is a product of all $U$-irreducible polynomials of degrees that divide $n$, i.e., 
\begin{displaymath}
a_n(t) = t^{q^n -(-1)^n} - 1 = \prod_{d \mid n}\prod_{\deg f(t)= d}f(t).
\end{displaymath}
Thus, $q^n - (-1)^n = \sum_{d\mid n}d\hat{\phi}_d(q)$, where $\hat{\phi}_d(q)$ is the number of $U$-irreducible polynomials over $\Fqq$, of degree $d$ (here $\phi_d(q)$ denotes the number of irreducible, in the usual sense, polynomials of degree $d$ over $\Fq$). Applying M\"obius inversion, we get the number of $U$-irreducible polynomials of degree $d$ over field $\Fqq$ to be
\begin{displaymath}
\hat{\phi}_n(q) = \frac{1}{n}\sum_{d\mid n}\mu\left(\frac{n}{d}\right)(q^d - (-1)^d).
\end{displaymath}
Hence we have the following lemma:
\begin{lemma}
$\hat{\phi}_n(q) = \phi_n(q)$, for $n \geq 3$. For $n = 1$ and $2$, we have
\begin{displaymath}
\hat{\phi}_1(q) = q + 1 \text{ and } \hat{\phi}_2(q) = \frac{q^2 - q - 2}{2}.
\end{displaymath}
\end{lemma}

\subsection{Similarity class types and centralizers.}
In \cite{Sh1}, we discussed that there is a bijection between similarity classes of $M_n(\Fq)$, and the following set of maps:
\begin{displaymath}
\{ \nu\colon \Irr(\Fq[t]) \rightarrow \Lambda \mid \text{$\sum_{f(t) \in \Irr(\Fq[t])} |\nu(f)|\deg(f) = n$.} \},
\end{displaymath}
where $\Lambda$ denotes the set of integer partitions. For unitary matrices, the characteristic polynomials are products of $U$-irreducible polynomials in $\Fqq[t]$. In case of $U_n(\Fq)$, we have a bijection between similarity classes of $U_n(\Fq)$, and the following set of maps:
\begin{displaymath}
\{ \nu\colon  U\Irr(\Fqq[t]) \rightarrow \Lambda \mid \text{$\sum_{f(t) \in \Irr(\Fq[t])} |\nu(f)| \deg(f) = n$.} \},
\end{displaymath}
where $U\Irr$ denotes the set of $U$-irreducible polynomials in $\Fqq[t]$. We can now bring in Green's definition (see~\cite{Gre}) of similarity class types, for the unitary groups. Before stating Green's definition of types, we define for non-negative integer $d$, and partition $\mu$, and a function $\nu\colon U\Irr(\Fqq[t]) \rightarrow \Lambda$, the number $r_\nu(\mu, d)$ as follows: 
\begin{displaymath}
r_\nu(\mu, d) = |\left \{f(t)\in U\Irr(\Fqq[t]) \mid \deg(f)= d,~\nu(f) = \mu \right\}|.
\end{displaymath}
\begin{defn}[Green]
Let $A$ and $B$ be two unitary matrices over $\Fqq$. Let $\nu_A$ denote the partition function associated with the similarity class of $A$, and $\nu_B$, be that associated with $B$. We say that $A$ and $B$ are of the same \emph{similarity class type} (or simply \emph{type}), if $r_{\nu_A}(\mu, d) = r_{\nu_B}(\mu, d)$ for all partitions $\mu$, and all $d \geq 0$.
\end{defn}
\noindent An alternative definition of similarity class types, which can be shown to be compatible with Green's definition is this. 
\begin{defn}
Let $(A_1, \ldots, A_k)$, and $(B_1, \ldots, B_l)$, be tuples of commuting unitary matrices. We say that these two tuples are of the same similarity class type if their respective common centralizers in $U_n(\Fq)$ are conjugate in $GL_n(\Fqq)$.
\end{defn} 
\noindent A type is written as a partition in the following way,  
\begin{displaymath}
\left (\lbd^{(1)}_1,\ldots, \lbd^{(1)}_{l_1} \right)_{d_1},  \ldots, \left(\lbd^{(j)}_1,\ldots, \lbd^{(j)}_{l_1}\right)_{d_j} ,
\end{displaymath}
where $\displaystyle\sum_{i = 1}^j |\lbd^{(i)}|d_i = n.$

For $n, k\geq 1$, we denote by $U_n(\Fq)^{(k)}$, the set of $k$-tuples of commuting matrices in $U_n(\Fq)$. We denote by $c_u(n,k,q)$, the number of simultaneous conjugacy classes of $U_n(\Fq)^{(k)}$. We can see that 
\begin{displaymath}
c_u(n,k,q) = \sum_{Z \sbteq U_n(\Fq)}s_Z c_Z(k-1),
\end{displaymath}
where $Z$ varies over the subgroups of $U_n(\Fq)$, the number $s_Z$ is the number of conjugacy classes of $U_n(\Fq)$ whose centralizer is conjugate to $Z$ in $GL_n(\Fqq)$, and $c_Z(k-1)$ is the number of orbits for the simultaneous conjugation action of $Z$ on $Z^{(k-1)}$. So, for example when $k=2$, given a matrix $A \in U_n(\Fq)$, to know the number of simultaneous conjugacy classes of pairs of commuting matrices, with $A$ as the first coordinate, it suffices to know the orbits for the conjugation action of the centralizer $\Zu{n}(A)$ on itself. In what follows we determine the branches for $U_2(\Fq)$ and $U_3(\Fq)$. In the similar way we define branches for the symplectic groups in later sections. 

\section{Branching rules for $U_2(\Fq)$}

In this section we explore the $2 \times 2$ unitary groups. There are a total of $q^2 + 2q$ similarity classes in $U_2(\Fq)$. We begin with explicit description of the canonical forms of the various types of classes in $U_2(\Fq)$. Recall that a matrix $A\in GL_2(\Fqq)$ is in $U_2(\Fq)$ if ${}^tA\beta \overline{A} = \beta$ for a fixed hermitian form $\beta$. In what follows we work with different hermitian form matrices while dealing with a particular class, depending on the one which makes our computations easier. We usually take the hermitian matrix to be either $I_2$, the identity matrix, or $\beta = \begin{pmatrix}  & 1\\ 1 &  \end{pmatrix}$. As these are equivalent over finite field, we can use any of these forms as per our convenience. Let $A = \begin{pmatrix} a_0 & a_1 \\ a_2 & a_3 \end{pmatrix} \in GL_2(\Fqq)$. Then, $A$ is unitary with respect to the hermitian form $I_2$ if it satisfies ${}^t\!A\overline{A}=I_2$. By writing $\overline{A}=\begin{pmatrix}a_0^q & a_1^q \\ a_2^q & a_3^q\end{pmatrix}$ we get the following equations to be satisfied by $A$, 
\begin{equation}\label{i1}
a_0^{q+1} + a_2^{q + 1} = 1,\ a_1^{q+1} + a_3^{q + 1} = 1,\ a_0a_1^q + a_2a_3^q = 0
\end{equation}
Similarly we get the following equations for $A\in GL_2(\Fqq)$ to be unitary with respect to the second $\beta$, 
\begin{equation}\label{b1}
a_0a_2^q + a_2a_0^q = 0,\ a_1a_3^q + a_3a_1^q = 0,\ a_0a_3^q + a_2a_1^q = 1. 
\end{equation}
Now we write down the $4$ types of classes and the canonical forms of matrices of each of the types of classes of $U_2(\Fq)$.
\begin{enumerate}
\item[$(1,1)_1$] This is $\Cent$ type corresponding to $(1,1)_1$. There are $q + 1$ such classes. This comprises of scalar matrices, $aI_2$. The matrices $aI_2 \in U_2(\Fq)$ with respect to either of the hermitian forms $I_2$ or $\beta$, if $a^{q+1} = 1$.
\item[$(2)_1$] There are $q+1$ such classes. In $GL_2(\Fqq)$, a matrix of a similarity class of this type can be written as $\begin{pmatrix} a_0 & a_1 \\  &a_0 \end{pmatrix}$. With respect to the second hermitian form $\beta$, such a matrix is unitary if it satisfies the equations~\ref{b1} which gives $a_0^{q+1} = 1$, and $a_0a_1^q + a_0^qa_1 = 0$. These equations have a non-zero solution as $q$ is an odd prime. Thus, we have a canonical form for a unitary matrix of the type $(2)_1$, of the form $\begin{pmatrix} a_0 & a_1 \\  & a_0 \end{pmatrix}$.
\item[$(1)_1(1)_1$] There are ${q+1}\choose 2$ such classes. A matrix of this type has diagonal canonical form $\begin{pmatrix} a_0 &  \\  & a_1 \end{pmatrix}$ in $GL_2(\Fqq)$, where $a_0 \neq a_1$. With respect to the hermitian form $I_2$, this diagonal matrix is in $U_2(\Fq)$ if and only if $a_0^{q+1}= a_1^{q+1} = 1$.
\item[$(1)_2$] There are $(q^2-q-2)/2$ such classes. The characteristic polynomial of a matrix of this type is an $U$-irreducible polynomial of degree $2$ in $\Fqq[t]$. Such a polynomial is of the form $(t-a)(t-a^{-q})$, where $a^{-q} \neq a$. The canonical form of a matrix of this type is 
$\begin{pmatrix} a &  \\  & a^{-q} \end{pmatrix}$.
\end{enumerate} 
\noindent Now we compute the branching rule for each type.
\begin{prop}\label{PropU211}
For a matrix $A$ of the $\Cent$ type, the branching rules are given as follows, 
\begin{center}
\begin{tabular}{cc} \hline 
 \text{Type of Branch} & \text{No. of Branches}\\ \hline
 $\Cent$ & $q + 1$\\ 
 $(2)_1$ & $q + 1$\\
 $(1)_1(1)_1$ & ${q+1}\choose 2$ \\
 $(1)_2$ & $(q^2 - q - 2)/2$. \\ \hline
 \end{tabular} \end{center}
\end{prop}
\begin{proof}
The centralizer of any matrix of the $\Cent$ type in $U_2(\Fq)$ is $U_2(\Fq)$ itself. Thus, enumerating the similarity classes in $U_2(\Fq)$ gives us the table mentioned in the statement of this proposition.
\end{proof}
\begin{prop}\label{PropU221}
For a matrix of similarity class type $(2)_1$, there are $q(q+1)$ branches of the type $(2)_1$.
\end{prop}
\begin{proof}
Let $A$ be a unitary matrix of type $(2)_1$. Then, $A$ has the canonical form $\begin{pmatrix} a_0 & a_1\\  & a_0 \end{pmatrix}$ where $a_0^{q+1} = 1$, and a fixed $a_1$, which satisfies $a_0a_1^q + a_0^qa_1 = 0$. The centralizer of $A$ in $U_{2}(\Fq)$ is $\left\{\begin{pmatrix}x_0 & x_1 \\  & x_0\end{pmatrix} \mid  x_0 \neq 0  \right\} \cap U_{2}(\Fq)$. In this case we work with the second hermitian form $\beta$, we have $x_0^{q +1} = 1$, and $x_1$ satisfies $x_0x_1^q + x_1x_0^q = 0$. For each $x_0$, there are $q$ such $x_1$. So, $\Zu{2}(A) = \left\{\begin{pmatrix}x_0 & x_1 \\  & x_0\end{pmatrix} \mid x_0^{q+1} = 1 {\rm \ and\ } x_0x_1^q + x_1x_0^q = 0 \right\}$. We note that $\Zu{2}(A)$ is a commutative subgroup of $U_{2}(\Fq)$ of size $q(q+1)$. Thus, each element of $\Zu{2}(A)$ is an orbit for the conjugation of $\Zu{2}(A)$ on itself. Each of these orbits has $\Zu{2}(A)$ as the centralizer. Thus, $A$ has $q(q+1)$ branches of the type $(2)_1$.
\end{proof}
\begin{prop}\label{PropU112}
A matrix of type $(1)_1(1)_1$ has $(q+1)^2$ branches of the type $(1)_1(1)_1$.
\end{prop}
\begin{proof}
A matrix $A$ of type $(1)_1(1)_1$ has the canonical form $\begin{pmatrix} a_0 & \\  & a_1 \end{pmatrix}$ where $a_0^{q+1} = a_1^{q+1} = 1$, and $a_0 \neq a_1$. This is a regular semisimple element with centralizer $\Zu{2}(A) = \left\{ \begin{pmatrix}x_0 &  \\  & x_1\end{pmatrix} \mid x_0^{q+1}= x_1^{q+1} = 1\right\}$, a maximal torus which is commutative. Thus $A$ has $(q+1)^2$ branches of the type $(1)_1(1)_1$.
\end{proof}
\begin{prop}\label{PropU12}
For type $(1)_2$, there are $(q^2-1)$ branches of type $(1)_2$.
\end{prop}
\begin{proof}
A unitary matrix $A$ of type $(1)_2$ has the canonical form $A = \begin{pmatrix}
a &  \\  & a^{-q} \end{pmatrix}$ where $a^{-q} \neq a$ with respect to the second hermitian form $\beta$. This is a regular semisimple class and hence centralizer is an anisotropic maximal torus which is commutative. This we can also do by explicit computation and get $\Zu{2}(A) = \left\{\begin{pmatrix}x_0 &  \\  & x_0^{-q}\end{pmatrix} \mid x_0 \neq 0 \right\}$. As $\Zu{2}(A)$ is commutative, each of its members form an orbit for the conjugation action of $\Zu{2}(A)$ on itself. Thus, $|\Zu{2}(A)| = q^2-1$ and $A$ has $q^2-1$ branches of type $(1)_2$.
\end{proof}

\subsection{Proof of Theorem~\ref{main1} in the case of $U_2(\Fq)$}\label{U2Main1}

We shall write down the four similarity class types in the order: $\left \{ (1,1)_1, (2)_1,(1)_1(1)_1,(1)_2 \right\}$. From Propositions~\ref{PropU211},~\ref{PropU221},~\ref{PropU112}, and~\ref{PropU12}, the branching rules for $U_2(\Fq)$ can be summarized in a matrix or in a table with rows and columns indexed by the similarity class types chosen above. We call this matrix $BU_2$ of which entries are described in the Table~\ref{tableI}.  Given a conjugacy class type $\tau$ of $U_2(\Fq)$, the branches of a canonical matrix of type $\tau$ are given in the column corresponding to $\tau$. 
Proof of the remaining part of this theorem for the case $U_3(\Fq)$ will follow in the next section.

\subsection{The number of simultaneous similarity classes}
We define $\bmf{1} = \begin{pmatrix}1&1&1&1\end{pmatrix}$  and  $\bmf{e}_1$ is the column vector with $1$ at the first place and $0$ elsewhere. From the branching matrix $BU_2$, we get $c_u(2,k,q) = \bmf{1}(BU_2)^k \bmf{e}_1$. Table~\ref{T2by2} shows some values of $c_u(2,k,q)$.
\begin{table} 
\caption{\text{$c_u(2,k,q)$}} \label{T2by2} 
\begin{tabular}{cc} \hline 
$k$ & $c_u(2,k,q)$ \\ \hline
$1$ & $q^{2} + 2 q + 1$ \\
$2$ & $q^{4} + 3 q^{3} + 5 q^{2} + 5 q + 2$\\
$3$ & $q^{6} + 4 q^{5} + 10 q^{4} + 17 q^{3} + 16 q^{2} + 7 q + 1$\\
$4$ & $q^{8} + 5 q^{7} + 17 q^{6} + 39 q^{5} + 53 q^{4} + 43 q^{3} + 23 q^{2} + 9 q + 2$\\ \hline
\end{tabular}
\end{table} 
We consider the generating series $h_u(2,t)$ for $c_u(2,k,q)$ in $k$:
\begin{displaymath}
h_u(2,t) = \sum_{k = 0}^\infty c_u(2,k,q)t^k.
\end{displaymath}
We get  
$$
h_u(2,t)  = \sum_{k = 0}^\infty  (\bmf{1} B_2^k \bmf{e}_1) t^k  = \sum_{k = 0}^\infty \bmf{1} B_2^kt^k \bmf{e}_1 = \bmf{1} \left(\sum_{k = 0}^\infty B_2^kt^k \right)\bmf{e}_1 = \bmf{1} (I_4 - tB_2)^{-1} \bmf{e}_1
$$ 
which after further simplification we get
\begin{displaymath}
h_u(2,t) = \frac{q^{4} t^{2} -  q^{3} t^{3} + 2 q^{3} t^{2} - 3 q^{2} t^{3} + q^{2} t^{2} - 3 q t^{3} - 2 q^{2} t -  t^{3} - 2 q t + 1}{(q t + t - 1)(q^{2} t -  t - 1) (q^{2} t + q t - 1)  (q^{2} t + 2 q t + t - 1)}.
\end{displaymath}
Thus
\begin{equation}\label{GenU2}
h_u(2,t) = \frac{1-2q(q-1)t + q^2(q+1)^2t^2 - (q-1)^3t^3}{(q t + t - 1)(q^{2} t -  t - 1) (q^{2} t + q t - 1)  (q^{2} t + 2 q t + t - 1)}.
\end{equation}
\noindent As both the numerator and denominator of $h_u(2,t)$ are polynomials with integer coefficients, it is clear that $c_u(2,k,q)$ is a polynomial in $q$ with integer coefficients. From the data in table~\ref{T2by2}, and calculations of $c_u(2,k,q)$ for higher $k$, upto 20 using Sage~\cite{SAGE}, we prove a part of Corollary\ref{CorMain1}(2).

We close this section with a comparison of $c_u(2,k,q)$ with $c_g(2,k,q)$, the number of simultaneous conjugacy classes of $k$-tuples of commuting $2\times 2$ matrices in $GL_2(\Fq)$.

\begin{proof}[\bf{Proof of Corollary~\ref{CorMain1}(1)}]
For $GL_2(\Fq)$, we have the following branching matrix (from the $2\times 2$ case in \cite{Sh1}) 
\begin{displaymath}
Bg_2 = \begin{pmatrix}
(1,1)_1& (2)_1&(1)_1(1)_1&(1)_2\\ \hline
	   q-1 & 0 & 0 & 0 \\ 
	   q-1 & q(q-1) & 0 & 0 \\
	   \frac{(q-1)(q-2)}{2} & 0 & (q-1)^2 & 0 \\
	   \frac{q(q-1)}{2} & 0 & 0 & q^2-1
 	   \end{pmatrix},
\end{displaymath}
and we have $c_g(2,k,q) = \bmf{1}.Bg_2^k. \bmf{e}_1$. This leads to the generating function, $h_G(2,t)$ in $k$ for $c_g(2,k,q)$:
\begin{equation}\label{GenGL2}
h_G(2,t) = \frac{1 - 2q(q-1)t + q^2(q-1)^2t^2 - (q-1)^3t^3}{(1-(q-1)t)(1-(q-1)^2t)(1-(q^2-q)t)(1-(q^2-1)t)}
\end{equation}
\noindent From Equation~\ref{GenU2} we get that $c_u(2,k,q)$ is:
\begin{equation*}
\begin{aligned}
c_u(2,k,q) &= \sum_{a+b+c+d = k}(q+1)^a (q+1)^b(q-1)^b (q^2+q)^c(q+1)^{2d} \\
&~ - 2(q^2+q)\left(\sum_{a+b+c+d = k-1}(q+1)^a (q+1)^b(q-1)^b (q^2+q)^c(q+1)^{2d}\right)\\
&~ + (q^2+q)^2\left(\sum_{a+b+c+d = k-2}(q+1)^a (q+1)^b(q-1)^b (q^2+q)^c(q+1)^{2d}\right)  \\
&~- (q+1)^3 \left(\sum_{a+b+c+d = k-3}(q+1)^a (q+1)^b(q-1)^b (q^2+q)^c(q+1)^{2d} \right)
\end{aligned}
\end{equation*}
\noindent and from Equation~\ref{GenGL2} we get that $c_g(2,k,q)$ is:
\begin{equation*}
\begin{aligned}
c_g(2,k,q) &= \sum_{a+b+c+d = k}(q-1)^a (q-1)^{2b} (q^2-q)^c(q+1)^d(q-1)^d \\
&~ - 2(q^2-q)\left(\sum_{a+b+c+d = k-1}(q-1)^a (q-1)^{2b} (q^2-q)^c(q+1)^d(q-1)^d\right)\\
&~ + (q^2-q)^2\left(\sum_{a+b+c+d = k-2}(q-1)^a (q-1)^{2b} (q^2-q)^c(q+1)^d(q-1)^d\right)  \\
&~- (q-1)^3 \left(\sum_{a+b+c+d = k-3}(q-1)^a (q-1)^{2b} (q^2-q)^c(q+1)^d(q-1)^d \right)
\end{aligned}
\end{equation*}
From the above we can see a certain duality between $c_u(2,k,q)$ and $c_g(2,k,q)$. We can obtain $c_u(2,k,q)$ from $c_g(2,k,q)$ by swapping the $(q-1)$'s and $(q+1)$'s in the formula of $c_g(2,k,q)$.
\end{proof}

\section{Branching rule for $U_3(\Fq)$}\label{S33}

Now, we consider $3 \times 3$ unitary group $U_3(\Fq)$, which is of size $q^3(q+1)(q^2-1)(q^3+1)$. The hermitian forms, which we will consider for various similarity classes in $U_3(\Fq)$ are 
\begin{displaymath}
\beta = \begin{pmatrix}&&1\\&1&\\1&&\end{pmatrix}, \ \ \ \gamma = \begin{pmatrix}&1&\\1&&\\&&1\end{pmatrix},\text{ and $I_3$ the $3\times 3$ identity matrix.}
\end{displaymath} 
There are $8$ types of similarity classes in $U_{3}(\Fq)$. We shall list them along with their canonical forms here. These can be obtained by doing a similar calculation as in the case of $U_2(\Fq)$ done in the last section from that of in $GL_3(\Fqq)$ with respect to one of the suitable forms mentioned above.
\begin{description}
\item[$(1,1,1)_1$] The scalar matrices $aI_3$, where $a^{q+1} = 1$.
\item[$(2,1)_1$] With respect to the hermitian form $\gamma$, a unitary matrix of this type has the canonical form $\begin{pmatrix} a_0 & a_1 & \\ &a_0 &  \\ &  & a_0 \end{pmatrix}$ where $a_0^{q+1} = 1$, and $a_1$ is some fixed member of $\Fqq$ that satisfies $a_0a_1^q + a_0^qa_1 = 0$.
\item[$(1,1)_1(1)_1$] The canonical form of a unitary matrix of this type is $\begin{pmatrix} a_0I_2 & \\ & a_1 \end{pmatrix}$ where $a_0^{q+1} = a_1^{q+1} = 1$, and $a_0 \neq a_1$.
\item[$(3)_1$] Unitary matrices of this type have the canonical form $\begin{pmatrix} a_0 & a_1 & a_2\\ & a_0 & a_1 \\ & & a_0 \end{pmatrix}$ where $a_0^{q+1} = 1$, $a_1, a_2$ are fixed elements of $\Fqq$ satisfying $a_0^qa_1 + a_0a_1^q = 0$ and $a_0^qa_2 + a_1^{q+1} + a_0a_2^q = 0$. 
\item[$(2)_1(1)_1$] Unitary matrices of this type have the canonical form $\begin{pmatrix} a_0 & a_1 & \\ & a_0 &  \\ & & b_0 \end{pmatrix}$ where $a_0^{q+1} = b_0^{q+1} = 1$, $a_0 \neq b_0$, and $a_1$ is some fixed element of $\Fqq$ satisfying $a_0a_1^q + a_0^qa_1 = 0$.
\item[$(1)_1(1)_1(1)_1$] Unitary matrices with respect to the hermitian form $I_3$ of this type have the canonical form $\begin{pmatrix} a_0 &  & \\ & b_0 &  \\ & & c_0 \end{pmatrix}$ where $a_0^{q+1} = b_0^{q+1} = c_0^{q+1} = 1$ and all $a_0, b_0, c_0$ are distinct.
\item[$(1)_2(1)_1$] Unitary matrices with respect to the hermitian form $\gamma$ have the canonical form $\begin{pmatrix} a &  & \\ & a^{-q} &  \\ & & b_0 \end{pmatrix}$ where $a^{-q} \neq a$, and $b_0^{q+1} = 1$.
\item[$(1)_3$] A matrix of this type has a $U$-irreducible polynomial of degree $3$, as its minimal polynomial. 
\end{description}
Now we move on to the branching rules for each type. 
\begin{prop}\label{PropU3111}
For a matrix $A$ of the $\Cent$ type, the branching rules are as follows,
\begin{center}
\begin{tabular}{cc|cc}\hline
\text{Type} & \text{No. of branches} &\text{Type} & \text{No. of branches} \\ \hline
$(1,1,1)_1$ & $q+1$ & $(2)_1(1)_1$ & $(q+1)q$ \\
$(2,1)_1 $ & $q+1$ & $(1)_1(1)_1(1)_1$ & ${q+1}\choose 3$ \\
$(1,1)_1(1)_1$ & $(q+1)q$ & $(1)_2(1)_1$ & $(q^2-q-2)(q+1)/2$\\
$(3)_1$ & $q+1$ & $(1)_3$ & $(q^3-q)/3$. \\ \hline
\end{tabular}
\end{center}
\end{prop}
\begin{proof}
The centralizer of any $\Cent$ type of matrix is the entire group $\U{3}$. Hence, enumerating the similarity classes gives the table above.
\end{proof}

\begin{prop}\label{PropU311}
For a matrix of type $(1,1)_1(1)_1$, the branching rules are as follows,
\begin{center}
\begin{tabular}{cc} \hline
\text{Type of Branch} & \text{No. of Branches} \\ \hline
$(1,1)_1(1)_1$ & $(q+1)^2$ \\
$(2)_1(1)_1$ & $(q+1)^2$\\
$(1)_1(1)_1(1)_1$ & $q(q+1)^2/2$ \\
$(1)_2(1)_2$ & $(q+1)(q^2 - q - 2)/2$. \\ \hline
\end{tabular}
\end{center} 
\end{prop}
\begin{proof}
A matrix $A$ of type $(1,1)_1(1)_1$ has the canonical form $\begin{pmatrix} a_0 & & \\ & a_0 & \\ & & b_0 \end{pmatrix}$, where $a_0 \neq b_0$, and $a_0^{q+1} = b_0^{q+1} = 1$. The centralizer of a matrix of this type is $\Zu{3}(A) = \left\{ \begin{pmatrix} X & \\ & z_0\end{pmatrix} \mid X \in U_{2}(\Fq), z_0^{q+1} = 1\right\}$. Thus, each of the branches of $A$ is of the type $\tau (1)_1$, where $\tau$ is a similarity class type from $\U{2}$. Hence the number of branches is the product of the number of classes $\U{2}$ of type $\tau$, and $q+1$. Thus, we get the required table.
\end{proof}

\begin{prop}\label{PropU321}
For a matrix of type $(2,1)_1$, there are $q(q+1)$ branches of type $(2,1)_1$, $q^2(q+1)$ branches of type $(2)_1(1)_1$, and $q^2-1$ branches of type $(3)_1$. 
\end{prop}
\begin{proof}
A matrix of this type $(2,1)_1$ has the canonical form $A = \begin{pmatrix} a & b & \\ & a & \\ & & a \end{pmatrix}$, where $a^{q+1} = 1$, and $b$ is a fixed element of $\Fqq$ satisfying $ab^q + a^qb = 0$. Thus, $\Zu{3}(A)$ is 
$$ \Zu{3}(A) = \left\{ \begin{pmatrix}a_0 & a_1 & -a_0dc^q \\ & a_0 & \\ & c & d\end{pmatrix} \mid a_0^{q+1} = d^{q+1} = 1, a_0a_1^q + c^{q+1} + a_0^qa_1 = 0 \right\}.
$$
Let $B = \begin{pmatrix}a_0 & a_1 & -a_0dc^q \\ & a_0 & \\ & c & d\end{pmatrix}$ and  $B' = \begin{pmatrix}a'_0 & a'_1 & -a'_0d'c'^q \\ & a'_0 & \\ & c' & d'\end{pmatrix}$ be in $\Zu{3}(A)$. Suppose $B'$ is a conjugate of $B$ by some matrix $X= \begin{pmatrix}x_0 & x_1 & -y^qx_0z_0 \\ & x_0 & \\ & y & z_0\end{pmatrix} \in \Zu{3}(A)$, i.e., $XB = B'X$. Equating the entries in $XB = B'X$ gives us $a'_0 = a_0$ and $d' = d$ and
\begin{eqnarray}
ya_0 + z_0c &=& yd + x_0c'\label{E211}\\
x_0a_1 - cx_0z_0y^q &=& x_0a'_1 - ya_0dc' \label{E212} 
\end{eqnarray}
We deal with the two cases separately when $a_0 \neq d$ and $a_0 = d$. 

{\bf When $a_0 \neq d$: } In equation~\ref{E211}, we can choose $y \in \Fqq$ such that $x_0c' = 0$, thus making $c' = 0$. Now replace $c$ by $c'= 0$, and substitute in equation~\ref{E212}. We get $x_0a_1 = x_0a'_1$, which implies $a'_1 = a_1$. Hence, $B$ is reduced to $B = \begin{pmatrix}a_0 & a_1 & \\ & a_0 & \\ & & d\end{pmatrix}$, and 
$$
\Zu{3}(A,B)= \left\{ \begin{pmatrix}x_0 & x_1 &\\ & x_0 & \\& & z_0\end{pmatrix} \mid x_0^{q+1}= z_0^{q+1}= 1, x_0x_1^q + x_1x_0^q = 0  \right\},
$$
which is the centralizer group of a unitary matrix of type $(2)_1(1)_1$. Hence $(A,B)$ is a branch of type $(2)_1(1)_1$. The number of such branches is $(q+1)q^2$ (as $a_0$ is arbitrary, $d \neq a_0$, and for each $a_0$, there are $q$ $a_1$'s satisfying $a_0a_1^q + a_1a_0^q = 0$).

{\bf When $a_0 = d$:} In this case, instead of using the $XB = B'X$ approach, and reducing $B$ to a possible `canonical' form, we shall look at $XB = BX$. Here we have $B$ with $a_0 = d$. We have,
\begin{displaymath}
\begin{pmatrix}
x_0 & x_1 & -y^qx_0z_0\\
 & x_0 &  \\
 & y & z_0
\end{pmatrix} \begin{pmatrix}
a_0 & a_1 & -c^qa_0^2\\
 & a_0 &  \\
 & c & a_0
\end{pmatrix} = \begin{pmatrix}
a_0 & a_1 & -c^qa_0^2\\
 & a_0 &  \\
 & c & a_0
\end{pmatrix}
\begin{pmatrix}
x_0 & x_1 & -y^qx_0z_0\\
 & x_0 &  \\
 & y & z_0
\end{pmatrix}.
\end{displaymath}
\noindent This leaves us with 
\begin{eqnarray}
c(x_0 - z_0) &=& 0 \label{E213}\\
cy^qx_0z_0 &=& yc^qa_0^2 \label{E214}
\end{eqnarray}
Let us first assume $c =0$. We have $B = \begin{pmatrix} a_0 & a_1 &  \\  & a_0 &  \\  &  & a_0 \\ \end{pmatrix}$, where $a_0a_1^q + a_1a_0^q = 0$. There are $q(q+1)$ such matrices. We have $\Zu{3}(A,B) = \Zu{3}(A)$. The class size is $\frac{q^3(q+1)^2}{q^3(q+1)^2} = 1$. Hence, number of orbits is $q(q+1)/1 = q(q+1)$. We have $q(q+1)$ branches of type $(2,1)_1$.

In the case when $c\neq 0$, we get $B = \begin{pmatrix} a_0 & a_1 & -c^qa_0^2 \\  & a_0 &  \\  & c & a_0 \end{pmatrix}$, where $a_0a_1^q + a_1a_0^q + c^{q+1} = 0$. There are $(q+1)q(q^2-1)$ such $B$'s. Now, from Equation~\ref{E213}, we have $x_0 = z_0$, and thus, Equation~\ref{E214} becomes $cx_0^2y^q = c^qa_0^2y$. Thus, $y^q = c^{q-1}a_0^2y/x_0^2$. There are $q$ such $y \in \Fqq$ so that $y^q = c^{q-1}a_0^2y/x_0^2$. Hence, the centralizer of $A$ and $B$ is:
\begin{displaymath}
\Zu{3}(A,B) = \left\{ \begin{pmatrix}
 x_0 & x_1 & -c^{q-1}a_0^2y \\
  & x_0 &  \\
  & y & x_0
\end{pmatrix} \begin{matrix}\text{ $x_0^{q+1} = 1$, $y^q = c^{q-1}a_0^2y/x_0^2$}\\\text{ $x_0x_1^q + x_1x_0^q + y^{q+1} = 0$ }\end{matrix}\right\}
\end{displaymath}
Hence, $|\Zu{3}(A,B)| = (q+1)q^2$. The centralizer is a commutative subgroup of $U_3(\Fq)$, and is conjugate to that of a matrix of type $(3)_1$. The size of the conjugacy class of $B$ in $\Zu{3}(A)$ is $\frac{q^3(q+1)^2}{q^2(q+1)} = q(q+1)$. Hence, the number of orbits is $\frac{q(q^2-1)(q+1)}{q(q+1)} = q^2 - 1$. So we have $q^2 -1$ branches of type $(3)_1$.
\end{proof}
For a matrix of the remaining types, $(3)_1$, $(2)_1(1)_1$, $(1)_1(1)_1(1)_1$, $(1)_2(1)_1$, and $(1)_3$, the centralizer group of the matrix is the group of polynomials in that matrix with coefficients from $\Fqq$, such that they are in $U_{3}(\Fq)$. Thus, for any such matrix $A$, the $\Zu{3}(A)$ is a commutative subgroup of $U_{3}(\Fq)$. Hence, each orbit for the conjugation action of $\Zu{3}(A)$ on itself, is a singleton. Thus, each branch $(A,B)$ of $A$ is of the same type as $A$. We summarize this in the following, 
\begin{prop}\label{PropU3Reg}
For a matrix $A$ of any of the types, $(3)_1$, $(2)_1(1)_1$, $(1)_1(1)_1(1)_1$, $(1)_2(1)_1$, and $(1)_3$, table below gives us type of $A$, type of the branch of $A$, and number of branches.
\begin{center}
\begin{tabular}{ccc} \hline
\text{Type of $A$} &\text{Type of branch} & \text{No. of branches}\\ \hline
$(3)_1$ & $(3)_1$ &$q^2(q+1)$ \\
$(2)_1(1)_1$ & $(2)_1(1)_1$ & $(q+1)^2q$ \\
$(1)_1(1)_1(1)_1$ & $(1)_1(1)_1(1)_1$ & $(q+1)^3$ \\
$(1)_2(1)_1$ & $(1)_2(1)_1$ & $(q^2-1)(q+1)$ \\
$(1)_3$ & $(1)_3$ & $(q^3 + 1)$   \\ \hline
\end{tabular}
\end{center}
\end{prop} 
\subsection{Proof of Theorem~\ref{main1} in the case of $U_3(\Fq)$} \label{U3Main1}
We shall write down the types of similarity classes of $U_{3}(\Fq)$ in the following order:
$$\left\{(1,1,1)_1, (2,1)_1, (1,1)_1(1)_1, (3)_1, (2)_1(1)_1, (1)_1(1)_1(1)_1, (1)_2(1)_1, (1)_3\right\}.
$$
Using the order above, and the results of Propositions~\ref{PropU3111}, \ref{PropU311}, \ref{PropU321}, and \ref{PropU3Reg}, we have the branching rules for $U_3(\Fq)$ summarized in the Table~\ref{tableII} with rows and columns indexed by the eight conjugacy class types of $U_3$ in the order mentioned above. We write the entries of this table as a matrix, called branching matrix, $BU_3$.
\noindent This completes the proof of this theorem. 

\subsection{The number of simultaneous conjugacy classes $c_u(3,k,q)$ and $c_g(3,k,q)$}
Now that we have the branching matrix $BU_3$, we define $\bmf{1} = \begin{pmatrix}1&1&1&1&1&1&1&1\end{pmatrix}$  and  $\bmf{e}_1 $ is the column vector with $1$ at the first place and $0$ everywhere else. Therefore
\begin{displaymath}
c_u(3,k,q) = \bmf{1}\cdot BU_3^k \cdot \bmf{e}_1.
\end{displaymath}
We list $c_u(3,k,q)$ for some $k$ in Table~\ref{T3by3}.
\begin{table}[h!] 
\caption{\text{$c_u(3,k,q)$ for $k = 1,2,3,4$.}} \label{T3by3} 
\begin{equation*}\begin{array}{cc} \hline 
k & c_u(3,k,q) \\ \hline
1 & q^{3} + 2 q^{2} + 3 q + 2 \\
2 & q^{6} + 3 q^{5} + 8 q^{4} + 15 q^{3} + 15 q^{2} + 8 q + 2 \\
3 & q^{9} + 4 q^{8} + 14 q^{7} + 37 q^{6} + 66 q^{5} + 81 q^{4} + 64 q^{3} + 29 q^{2} + 7 q + 1 \\
4 & q^{12} + 5 q^{11} + 22 q^{10} + 74 q^{9} + 178 q^{8} + 313 q^{7} +  395q^{6} + 357 q^{5} +  241 q^{4} + 126 q^{3} + \\ & 49 q^{2} + 13 q + 2 \\ \hline
\end{array}\end{equation*}
\end{table} 
From the data in the this table, and further computations using Sage~\cite{SAGE} for $c_u(3,k,q)$ up to $k = 20$, we see that $c_u(3,k,q)$ is a polynomial in $q$ with non-negative integer coefficients. This proves Corollary~\ref{CorMain1}(2). We recall the branching matrix of $GL_3(\Fq)$ from~\cite{Sh1},
\begin{displaymath}
 B_{g_3} = \left(\begin{smallmatrix}
(1,1,1)_1 & (2,1)_1 & (1,1)_1(1)_1& (3)_1 & (2)_1(1)_1 & (1)_1(1)_1(1)_1 & (1)_2(1)_1& (1)_3\\ \hline 
&&&&&&& \\ 
q - 1 & 0 & 0 & 0 & 0 & 0 & 0 & 0 \\
q - 1 & q(q-1) & 0 & 0 & 0 & 0 & 0 & 0 \\
(q-1)(q-2) & 0 & (q-1)^2 & 0 & 0 & 0 & 0 & 0 \\
q - 1 & q^{2} - 1 & 0 & q^2(q-1) & 0 & 0 & 0 & 0 \\
(q-1)(q-2)& (q-1)(q-2)q & (q-1)^2 & 0 & q(q-1)^2 & 0 & 0 & 0 \\
{q-1 \choose 3} & 0 & (q-1){q-1 \choose 2} & 0 & 0 & (q-1)^3 & 0 & 0 \\
(q-1){q \choose 2}& 0 & (q-1){q \choose 2} & 0 & 0 & 0 & (q-1)^2(q+1) & 0 \\
\frac{(q^3-q)}{3} & 0 & 0 & 0 & 0 & 0 & 0 & q^{3} - 1
\end{smallmatrix}\right).
\end{displaymath}
The number of types for $GL_3(\Fq)$ and $U_3(\Fq)$ are same and given by the same partitions. With the help of the branching matrices $BU_3$ (for $U_3(\Fq)$), and $B_{g_3}$ (for $GL_3(\Fq)$) we can compute $c_u(3,k,q)$ and $c_g(3,k,q)$. However, we do not see any obivious interesting pattern like the duality pattern visible between $c_u(2,k,q)$, and $c_g(2,k,q)$.

\section{Branhcing Rules of $Sp_{2}(\Fq)$}\label{branchsp2} 

Here we will see the branching rules of $2 \times 2$ symplectic matrices over the finite field $\Fq$. We recall that $Sp_2(\Fq) = \{ A \in GL_2(\Fq) \mid A.\beta. {}^t\!A = \beta \}$ where $\beta = \begin{pmatrix}  & 1 \\ -1 & \end{pmatrix}$. A simple calculation shows that $Sp_2(\Fq)$ is the special linear group $SL_2(\Fq)$. Although the conjugacy classes are well known in this case, we list them in Appendix~\ref{appendixa}, with the class types being denoted in a way similar to the class type notations in~\cite{Sr}. 

\begin{prop}\label{Sp2Pcent}
For a matrix of type $C$ (the $\Cent$ type), the branching rules are as follows,  
 \begin{center}
\begin{tabular}{cc} \hline
    \text{ Type of Branch} & \text{No. of Branches} \\ \hline 
            $C$ & $2$ \\
           $ A_1$ & $2$ \\
           $ A_2$ & $2$ \\
           $ D$ & $\frac{q-3}{2}$ \\
            $Ir$ & $\frac{q-1}{2}$. \\ \hline
\end{tabular}
\end{center}
\end{prop}
\begin{proof}
 As the matrix is central, its centralizer is the whole of $Sp_2(\Fq)$, thus the result follows.
\end{proof}

\noindent The other four types $A_1$, $A_2$, $D$, and $C$ are regular types, hence for a matrix $A$ of any of these four types, its centralizer $Z_{Sp_2(\Fq)}(A)$ is polynomial in $A$ and hence it is a commutative group. Thus the commuting pair $(A,B)$ is of the same type as $A$. Thus, from the above argument, and the centralizer data listed in the same table, we have
\begin{prop}\label{Sp2Preg}
The table below shows the type, its only branch type, and the number of branches, for the regular types.
\begin{center}
\begin{tabular}{ccc} \hline
 \text{Type} & \text{Branch Type} & \text{No. of Branches} \\ \hline
 $A_1$ & $A_1$ & $2q$ \\
$A_2$ & $A_2$ & $2q$ \\ 
$ D$ &$ D$ & $q-1$ \\
$ Ir$ & $Ir$ & $q + 1$. \\ \hline
\end{tabular}
\end{center}
\end{prop}
\subsection{Proof of Theorem~\ref{Main2} for $Sp_2(\Fq)$}\label{SecSp2} 
From Propositions~\ref{Sp2Pcent}~and~\ref{Sp2Preg}, we have the branching rules for $Sp_2(\Fq)$ summarized in the Table~\ref{tableIII} which we call the matrix $BSp_2$, indexed by the rows and columns in the order $ \{C,A_1,A_2,D,Ir\}$.
The proof of remaining part of the theorem is in the next section.

\subsection{Proof of Corollary~\ref{CorMain1}(2)} Denoting the number of simultaneous conjugacy classes of $k$-tuples of commuting matrices in $Sp_2(\Fq)$ by $c_s(2,k,q)$, we have
\begin{displaymath}
 c_s(2,k,q) = \begin{pmatrix} 1&1&1&1&1\end{pmatrix}.BSp_2. e_1
\end{displaymath}
where $e_1$ is the column vector with $1$ at the first place and $0$ elsewhere.
We list some values of $c_s(2,k,q)$.
\begin{center}
\begin{tabular}{cc}\hline
 $k$ & $c_s(2,k,q)$ \\ \hline
$1$ &$ q+ 4$ \\
$2$ & $q^2 +  8q + 9$ \\
$ 3$ & $q^3 + 16q^2 + 19q + 16$  \\
$ 4$ & $q^{4} + 32 q^{3} + 38 q^{2} + 32 q + 33$ \\ \hline
\end{tabular}
\end{center}
We calculated $c_s(2,k,q)$ for $k$ up to $20$ using Sage~\cite{SAGE}, and observed that for each of those $k$, the value of $c_s(2,k,q)$ is a polynomial in $q$ with non-negative integer coefficients. This proves Corollary~\ref{CorMain1}(2).

\section{Branching Rules for $Sp_4(\Fq)$} 

We now move to $Sp_4(\Fq)$. We know that $Sp_4(\Fq)$ is $\{A \in GL_4(\Fq) \mid A.{\beta}. {}^t\!A = \beta \}$ where $\beta$ is a non-degenerate skew-symmetric matrix. We use one of the following depending on the ease of computation,
\begin{displaymath}
\beta_1 = \begin{pmatrix}  & 1 &  &  \\ -1 &  &  &  \\  &  &  & 1 \\  &  & -1 &  \end{pmatrix}, \text{ or }         
\beta_2 = \begin{pmatrix}  &  &  & 1 \\  &  & 1 &  \\  & -1 &  &  \\ -1 &  &  & \end{pmatrix} \text{, or }
\beta_3 = \begin{pmatrix}  & I_2\\  -I_2 &  \end{pmatrix}.
\end{displaymath}
This group is of size $q^4(q^4-1)(q^2-1)$. The conjugacy classes of $Sp_4(\Fq)$ are neatly described in Srinivasan's paper~\cite{Sr} (see pages 489-491) with respect to the form $\beta_1$. We reproduce the same in Appendix~\ref{appendixa} in Table~\ref{TSp4Con} for two reasons (a) for the convenience of reader, and, (b) a slight change in the nomenclature of the conjugacy class types, but based very much on Srinivasan's nomenclature. 

\begin{prop}\label{PA1}
For a matrix of type $A_1$, the branching rules are the first two columns of the Table~\ref{TSp4Con}.
\end{prop}
\begin{proof}
Since these matrices are central, the result follows.
\end{proof}
\begin{prop}\label{new1}
 For a matrix in $Sp_4(\Fq)$ of type $A_2$, there are:
 \begin{itemize}
  \item $2q$ branches of type $A_2$.
  \item $4(q-1)$ branches of type $A_4$.
  \item $2q$ branches of type $D_2$.
  \item $4q$ branches of type $D_3$.
  \item $q(q-3)$ branches of type $C_4$.
  \item $q(q-1)$ branches of type $C_2$.
  \item $4q$ branches of a new type, which we shall call $N_1$. 
  The common centralizer of a pair of this type is $\left\{ \begin{pmatrix}x_0 &  & y_2 & x_1 \\  & z_0 &z_1 & y_2z_0x_0 \\  &  & z_0 &  \\  &  &  & x_0\end{pmatrix} \mid x_0^2 = z_0^2 = 1 \right\}$. 
  This centralizer is of size $4q^3$, which none of the already known centralizers have.
  \item $4$ branches of a new type, which we shall call $N_2$. A pair of this type has the common centralizer $   \left\{ \begin{pmatrix}x_0 & y_1 &  & x_1 \\  & x_0 & &  \\  & z_2 & x_0 & -y_1 \\  &  &  & x_0\end{pmatrix}\mid x_0^2 =  1  \right\}$,  which is of size $2q^3$, which none of the already known centralizers have.
  \end{itemize}
\end{prop}
\begin{proof} A matrix of type $A_2$ has the canonical form $A=\begin{pmatrix} a_0 &  &  & \lambda \\  & a_0 &  & \\  &  & a_0 & \\ &&& a_0\end{pmatrix}$, where $a_0 = \pm1$ and $\lambda= 1$ or $\gamma$ (where $\langle \gamma \rangle = \Fq^*$, hence $\gamma$ is a non-square), with respect to the form $\beta_2$. To prove this proposition, it is enough to prove it, when $\lambda = 1$. 
The centralizer of $A$ is
$$
Z_{Sp_4(\Fq)}(A) = \left\{\begin{pmatrix}a_0 & b_1 & b_2 &a_1\\  & c_0 & c_1 & a_0(b_2c_0-b_1c_1)\\ & c_2&c_3& a_0(b_2c_2 - b_1c_3)\\  &  &  & a_0\end{pmatrix} \mid a_0 = \pm 1,  \left(\begin{smallmatrix}c_0 & c_1 \\c_2 & c_3\end{smallmatrix}\right) \in SL_2(\Fq)\right\}.
$$
We let $B = \begin{pmatrix}a_0 & b_1 & b_2 &a_1\\  & c_0 & c_1 & a_0(b_2c_0-b_1c_1)\\ & c_2&c_3& a_0(b_2c_2 - b_1c_3)\\  &  &  & a_0\end{pmatrix}$, $X = \begin{pmatrix}x_0 & y_1 & y_2 &x_1\\  & z_0 & z_1 & x_0(y_2z_0-y_1z_1)\\ & z_2&z_3& x_0(y_2z_2 - y_1z_3)\\  &  & & x_0\end{pmatrix}$, and $B' = \begin{pmatrix}a'_0 & b'_1 & b'_2 &a'_1\\  & c'_0 & c'_1 & a'_0(b'_2c'_0-b'_1c'_1)\\ & c'_2&c'_3& a'_0(b'_2c'_2 - b'_1c'_3)\\  &  &  & a'_0\end{pmatrix}$.
Like in the proofs of any of the propositions in~\cite{Sh1}, we work on $XB = B'X$, to reduce $B$ to the simplest forms possible, and find the common centralizer of the branch $(A,B)$ of $A$, accordingly, and state the type and number of branches. In this case $XB = B'X$ leads us to $a'_0 =a_0$, and $\left(\begin{smallmatrix}z_0 & z_1 \\z_2 & z_3\end{smallmatrix}\right)\left(\begin{smallmatrix}c_0 & c_1 \\c_2 & c_3\end{smallmatrix}\right) = \left(\begin{smallmatrix}z_0 & z_1 \\z_2 & z_3\end{smallmatrix}\right)\left(\begin{smallmatrix}c'_0 & c'_1 \\c'_2 & c'_3\end{smallmatrix}\right)$.  Thus, we can take $\left(\begin{smallmatrix}c_0 & c_1 \\c_2 & c_3\end{smallmatrix}\right)$, to be a representative of a conjugacy class of $SL_2(\Fq)$, and $\left(\begin{smallmatrix}z_0 & z_1 \\z_2 & z_3\end{smallmatrix}\right)$ to be its centralizer matrix.  This leads us to the following set of equations:
\begin{eqnarray*}
 x_0(b_1, b_2) + (y_1, y_2) \left(\begin{smallmatrix}c_0 & c_1 \\c_2 & c_3\end{smallmatrix}\right) &=& (b_1, b_2)\left(\begin{smallmatrix}z_0 & z_1 \\z_2 & z_3\end{smallmatrix}\right) + a_0(y_1,  y_2).\\
 x_0a_1 + (y_1, y_2)\begin{pmatrix} a_0(b_2c_0 - b_1c_1)\\ a_0(b_2c_2  - b_1c_3)\end{pmatrix} &=& x_0a'_1 +(b'_1, b'_2)\begin{pmatrix}x_0(y_2z_0-y_1z_1)\\x_0(y_2z_2 - y_1z_3)\end{pmatrix}. 
\end{eqnarray*}
We use these to simplify $B$ to get the branches listed in the statement.
\end{proof}
\begin{prop}\label{new2}
For a matrix of type $A_3$ the branches are as follows:
\begin{itemize}
 \item $2q$ branches of type $A_3$.
 \item $\displaystyle\frac{(q-3)q}{2}$ branches of type $B_9$.
 \item $2q^2$ branches of type $D_3$.
 \item $q(q+3)$ branches of a new type, which we will call new type $N_3$, 
 with the following common centralizer $\left\{ \begin{pmatrix}x_0 &  & x_1 &y_1 \\  & x_0 &y_1 & y_2 \\  &  & x_0 &  \\  &  &  & x_0\end{pmatrix} \mid x_0^2 = 1\right\}$. 
 \item $2q(q-1)$ branches of the new type, $N_1$.
\end{itemize}
\end{prop}
\begin{proof}
A matrix $A$ of type $A_3$ has the canonical form $\begin{pmatrix}
  a_0 &  & 1 &  \\
  & a_0 &  & -1 \\
   &  & a_0 & \\
   &  &  & a_0
  \end{pmatrix}$, where $a_0 = \pm1$, with respect to the bilinear form $\beta_3$. The centralizer of $A$ is  
  $$
  Z_{Sp_4(\Fq)}(A) = \displaystyle\left\{ \begin{pmatrix} C & D.^tC^{-1}\\  & \Delta C\Delta  \end{pmatrix} \mid D=~^tD, C.\Delta.^tC = \Delta, D\in M_2(\Fq) \right\}$$
where $\Delta= \begin{pmatrix}1 &  \\ &-1\end{pmatrix} =\Delta^{-1}$.   
So, the $C$ described above is a matrix in the orthogonal group $O_2^+(\Fq)$. We see $C$ satisfying: $^tC^{-1} = \Delta C\Delta $.

 Let $B = \begin{pmatrix} C & D.{}^t\!C^{-1}\\ & \Delta  C\Delta  \end{pmatrix}$, and $X = \begin{pmatrix} Z & Y.{}^t\!Z^{-1}\\ & \Delta Z\Delta  \end{pmatrix}$, and let $B' = \begin{pmatrix} C' & D'.^t\!C'^{-1}\\ & \Delta C'\Delta  \end{pmatrix}$. We work on $XB = B'X$, to reduce $B$ to the canonical forms of the branches of $A$. Here, $XB = B'X$ leads us firstly to $ZC  = C'Z$. So we might as well take $C$ to be a representative of a conjugacy class of $O_2^+(\Fq)$, and $Z$ to be a matrix in its centralizer in $O_2^+(\Fq)$. With this, and $XB = B'X$, we have the equation $Z.D\Delta C\Delta  + Y\Delta Z\Delta . \Delta C\Delta  = C.Y\Delta Z\Delta  + D'\Delta C\Delta . \Delta Z\Delta$ which further simplifies to $ Z.D\Delta C\Delta  + Y\Delta ZC\Delta = C.Y\Delta Z\Delta  + D'\Delta CZ\Delta$. Using the various conjugacy classes of $O_2^+(\Fq)$ (mentioned in Appendix~\ref{appendixb}), we get the branches as mentioned in the statement of this proposition.
\end{proof}
\begin{prop}\label{new3}
For a matrix of type $A_3'$ the branches are,
\begin{itemize}
 \item $2q$ branches of type  $A'_3$.
 \item $2q^2$ branches of type $D_3$.
 \item $\displaystyle\frac{q(q-1)}{2}$ branches of type  $B_7$.
 \item $q(q-1)$ branches of the new type $N_3$, and 
 \item $2q(q-1)$ branches of new type $N_1$.
\end{itemize}
\end{prop}
\begin{proof}
A matrix $A$ of this type has a canonical form $\begin{pmatrix} a_0 I_2 &\Delta&\\
&a_0I_2  \end{pmatrix}$  with respect to the bilinear form $\beta_3$ where $I_2$ is the identity matrix, $\Delta=\begin{pmatrix}1& \\ & -\gamma\end{pmatrix}$with $a_0 = \pm1$, and $\gamma$ is such that $\langle \gamma \rangle = \F_q^*$, a non-square. The centralizer of $A$ is
$$ Z_{Sp_4(\Fq)}(A) =\displaystyle\left\{ \begin{pmatrix} C & D.^tC^{-1}\\ & \Delta^{-1} C\Delta \end{pmatrix} \mid D \in M_2(\Fq), D=~^tD,  C.\Delta .{}^tC = \Delta \right\}.
$$
The matrix $C$ lies in the non-split orthogonal group $O_2^-(\Fq)$. Using the steps similar to the proof of Proposition~\ref{new2}, we get the branching rules as per the statement of the lemma, with help of the conjugacy classes of $O_2^-(\Fq)$ described in Appendix~\ref{appendixb}.
\end{proof}
\begin{prop}
 A matrix of type $B_6$, has the following branches:
 \begin{itemize}
  \item $q+1$ branches of type $B_6$,
  \item $q(q-2)$ branches of type $B_2$,
  \item $q+1$ branches of type $B_7$, and
  \item $q+1$ branches of type $B_4$.
 \end{itemize}
\end{prop}
\begin{proof}
 A matrix of this type, in $Sp_4(\Fq)$ has canonical form $A = \begin{pmatrix} & -1 &  &  \\  1 & b &  &  \\ &  &  & 1 \\ &  & -1 & b \end{pmatrix}$, where $x^2 - bx + 1$ is an irreducible polynomial in $\Fq[x]$ with respect to the form $\beta_2$. This is different from the way the canonical form of this is given in the Table~\ref{TSp4Con}, as the form in the table has entries from the degree $2$ extension $\Fqq$ of $\Fq$. The centralizer in $Sp_4(\Fq)$ is 
\begin{displaymath}
 Z_{Sp_4(\Fq)}(A) = {\bf B} ~\sqcup~\bigsqcup_{w\neq 0}Cl(w).{\bf B}~\sqcup~AD(1).{\bf B},
\end{displaymath}
where $\bf B$ is the subgroup
\begin{displaymath}
 \left\{ \left(\begin{smallmatrix} 1 &  & -by/2 & y \\  & 1 & y & -by/2\\  &  & 1 & \\ & & & 1  \end{smallmatrix}\right).\left(\begin{smallmatrix} x_0 & -x_1 &  &  \\ x_1 & x_0 + bx_1 &  & \\  &  & z_0 & z_1\\ & & -z_1 & z_0 + bz_1  \end{smallmatrix}\right) \mid x_0z_0 + bx_0z_1+ x_1z_1 = 1, x_0z_1 = x_1z_0 \right\},
\end{displaymath}
$Cl(w) = \begin{pmatrix}1& & & \\ & 1& &  \\ bw/2 & w & 1 &  \\ w & bw/2 &  & 1 \end{pmatrix}$ and $AD(1)=\begin{pmatrix}    &  & 1 &  \\   &  &  & -1 \\   1 &  &  & \\    & -1 &  & 
\end{pmatrix}$. 
The subgroup $\bf B$ is of size $q(q^2-1)$, and putting the $(q-1 + 2) = q+1$ cosets together, the $Z_{Sp_4(\Fq)}(A)$ is of size $q(q+1)(q^2-1)$.
The center of $Z_{Sp_4(\Fq)}(A)$ is $\left\{ \left(\begin{smallmatrix}   a_0&-a_1&&\\ a_1&a_0+ba_1 &  &\\  & &a_0 & a_1\\   & &a_1&a_0 + ba_1  \end{smallmatrix}\right) \mid  a_0^2 + ba_0a_1 + a_1^2 = 1\right\}.$
The centralizer of each of these in $Z_{Sp_4(\Fq)}(A)$ is $Z_{Sp_4(\Fq)}(A)$ itself, and these are $q+1$ in number. Hence, $(q+1)$ branches of yype $B_6$. 

Now, we take a matrix $B = \begin{pmatrix} a_0&-a_1&&\\ a_1&a_0+ba_1 & &\\ &&b_0 & b_1\\ &&b_1&b_0 + bb_1 \end{pmatrix}$ with $a_0^2+ba_0a_1 + a_1^2 \neq 1 $ then the $Z_{Sp_4(\Fq)}(A,B)$ is
$$\left\{ \left(\begin{matrix} x_0 & -x_1 & & \\ x_1 & x_0 + bx_1 & &\\ & & z_0 & z_1\\ && -z_1 & z_0 + bz_1  \end{matrix}\right) \mid x_0z_0 + bx_0z_1 + x_1z_1 = 1, x_0z_1 = x_1z_0  \right\}
 $$ 
which is of size $q^2-1$. This is a branch of type $B_2$, and there are $q(q-2)$ of these.
Next, we take $B = \begin{pmatrix}1 & & -b/2 & 1 \\ & 1 & 1 & -b/2 \\ & & 1 &  \\ & & & 1\end{pmatrix}\left(\begin{matrix}   a_0&-a_1&&\\ a_1&a_0+ba_1 &  &\\  & &a_0 & a_1\\   & &a_1&a_0 + ba_1  \end{matrix}\right)$, where $a_0^2+a_0a_1+a_1^2 = 1$, and then 
$$
Z_{Sp_4(\Fq)}(A, B) =   \left\{ \left(\begin{smallmatrix} 1 & & -by/2 & y \\  & 1 & y & -by/2\\  &  & 1 &\\ & &  & 1  \end{smallmatrix}\right).\left(\begin{smallmatrix} x_0 & -x_1 &  &  \\ x_1 & x_0 + bx_1 &  & \\  &  & x_0 & x_1\\ & & -x_1 & x_0 + bx_1  \end{smallmatrix}\right) \mid x_0^2 + bx_0x_1+ x_1^2 = 1 \right\}.
$$
This centralizer is commutative of size $q(q+1)$. This branch $(A,B)$ of $A$ is of type $B_7$, and there are $(q+1)$ such branches.  

Lastly, we take $B = AD(1)\left(\begin{smallmatrix}   a_0&-a_1&&\\ a_1&a_0+ba_1 &  &\\  & &a_0 & a_1\\   & &a_1&a_0 + ba_1  \end{smallmatrix}\right)$, where $a_0^2 + ba_0a_1 = a_1^2 = 1$. The centralizer 
$Z_{Sp_4(\Fq)}(A,B)$ is 
$$ \left\{ \left(\begin{smallmatrix} x_0 & -x_1 &  &  \\ x_1 & x_0 + bx_1 &  & \\  &  & x_0 & x_1\\ & & -x_1 & x_0 + bx_1  \end{smallmatrix}\right)\left(\begin{smallmatrix} y_0 &  &  \frac{-by_1}{2}& y_1 \\  & y_0 & y & \frac{-by_1}{2}\\ \frac{-by_1}{2} & -y_1 & y_0 & \\ -y_1& \frac{-by_1}{2}&  & y_0  \end{smallmatrix}\right) \mid \begin{matrix}x_0^2 + bx_0x_1+ x_1^2 = 1, \\  y_0^2 - \frac{(b^2-4)y_1^2}{4}= 1 \end{matrix} \right\}.
$$
This is a group of size $(q + 1)^2$. This branch $(A,B)$ is of type $B_4$, and there are $q+1$ branches. This completes the proof.
\end{proof}
\begin{prop}
 For a symplectic matrix of type $B_8$, there are:
 \begin{itemize}
  \item $(q-1)$ branches of type $B_8$.
  \item $(q-1)$ branches of type $B_9$.
  \item $\displaystyle\frac{1}{2}(q-1)(q-2)$ branches of type $B_3$.
  \item $\displaystyle\frac{1}{2}q(q-1)$ branches of type $B_2$.
 \end{itemize}
\end{prop}
\begin{proof} A matrix $A$ of type $B_8$ has the canonical form $\begin{pmatrix}aI_2 &  \\ & a^{-1}I_2\end{pmatrix}$, where $a \in \Fq^*$ and $a \neq \pm 1$, with respect to the form $\beta_3$. The centralizer of $A$ is: 
$\displaystyle\left\{\begin{pmatrix} C &  \\  & ^tC^{-1}\end{pmatrix}\mid C \in GL_2(\Fq) \right\}$. It is enough to know the conjugacy classes of $GL_2(\Fq)$ here. 
So, each branch is of the form $B = \begin{pmatrix} C & \\ & ^tC^{-1}\end{pmatrix}$, where $C$ is a canonical matrix of its conjugacy class type, and $Z_{Sp_4(Fq)}(A,B) = \left\{ \begin{pmatrix} X&\\&^tX^{-1}\end{pmatrix} \mid XC = CX\right\}$. 
\end{proof}
The next four non-regular types $C_1$,$C_3$, $D_1$, and $D_2$ are non-primary types, i.e., their canonical forms are of the kind: $\begin{pmatrix}A & \\ & B\end{pmatrix}$, where $A$ and $B$ are non-conjugate elements of $Sp_2(\Fq)$. Thus its centralizer is
$$
Z_{Sp_4(\Fq)}\left(\begin{pmatrix}A & \\ & B\end{pmatrix}\right) = \left\{\begin{pmatrix}X& \\ & Z\end{pmatrix}\mid X\in Z_{Sp_2(\Fq)}(A), Z\in Z_{Sp_2(\Fq)}(B)  \right\}.
$$
Hence the branches of such a matrix are of the form $\begin{pmatrix}C & \\ & D\end{pmatrix}$, where $C$ is a branch of $A$, and $D$ is a branch of $B$. Thus we can use this argument to prove the next four propositions regarding the branching of the types $C_1$, $C_3$, $D_1$, and $D_2$. 
\begin{prop}
 For a symplectic matrix of type $C_1$, these are the branches:
 \begin{itemize}
  \item $2(q+1)$ branches of type $C_1$.
  \item $4(q+1)$ branches of type $C_2$.
  \item $\displaystyle\frac{1}{2}(q-3)(q+1)$ branches of type $B_5$.
  \item $\displaystyle\frac{1}{2}(q^2-1)$ branches of type $B_4$.
 \end{itemize}
\end{prop}

\begin{prop}
 For a symplectic matrix of type $C_3$, these are the branches:
 \begin{itemize}
\item $2(q-1)$ branches of type $C_3$.
\item $4(q-1)$ branches of type $C_4$.
\item $\displaystyle\frac{1}{2}(q-3)(q-1)$ branches of type $B_3$.
\item $\displaystyle\frac{1}{2}(q-1)^2$ branches of type $B_5$.
\end{itemize}
\end{prop}

\begin{prop}
 For a symplectic matrix of type $D_1$, the branching data is as follows:
 \begin{center}
   \begin{tabular}{cc|cc} \hline
    \text{Branch Type} & \text{No. of Branches} & \text{Branch Type} & \text{No. of Branches} \\ \hline
    $B_3$ & $\frac{(q-3)^2}{4}$ & $C_3$ & $2(q-3)$ \\
    $B_4$ & $\frac{(q-1)^2}{4}$ & $C_4$ & $4(q-3)$  \\
    $B_5$ & $\frac{1}{2}(q-1)(q-3)$  &  $D_1$ & 4   \\
    $C_1$ & $2(q-1)$ & $D_2$ & 16\\
    $C_2$ & $4(q-1) $ &    $D_3$ & 16.  \\ \hline 
 \end{tabular}
\end{center}
\end{prop}

\begin{prop}
 A symplectic matrix of type $D_2$ has the following branches:
\begin{itemize}
\item $4q$ branches of type $D_2$,
\item $8q$ branches of type $D_3$,
\item $q(q-3)$ branches of type $C_4$, and
\item $q(q-1)$ branches of type $C_2$.
\end{itemize}
\end{prop}

\noindent These are the branching rules of the non-regular types. The remaining types of similarity classes of $Sp_4(\Fq)$ are the regular types. So, their centralizers in $Sp_4(\Fq)$ are commutative subgroups. Each of these matrices have the same regular type as their only branch. We summarize this in the following,
\begin{prop}
For all the $\Reg$ types, the only branch is that type itself, and the the number of branches is the size of its centralizer. The branching for the 11 $\Reg$ types in $Sp_4(\Fq)$ is as follows:
\begin{center}
\begin{tabular}{cc|cc} \hline
    $\Reg$\text{ Type} & \text{No. of Branches} & $\Reg$\text{ Type} & \text{No. of Branches} \\ \hline
    $A_4$ & $2q^2$ & $B_7$ & $q(q+1)$ \\
    $B_1$ & $q^2 + 1$ & $B_9$ & $q(q-1)$ \\
    $B_2$ & $q^2-1$ & $C_2$ & $2q(q+1)$ \\
    $B_3$ & $(q-1)^2$ & $C_4$ & $2q(q-1)$ \\
    $B_4$ & $(q+1)^2$ &  $D_3$ & $4q^2$. \\ 
    $B_5$ & $q^2-1$ &&\\ \hline
\end{tabular}
\end{center}
\end{prop}

\subsection{Branching rules of the new types in $Sp_4(\Fq)$}
In the process of computing the branching rules for $A_2, A_3$ and $A_3'$ in the Proposition~\ref{new1},~\ref{new2} and~\ref{new3} we get three new types $N_1, N_2$ and $N_3$. Thus to complete the branching rule for arbitrary size tuples we further need branching rules for these new ones.
\begin{prop}
 For a commuting pair of symplectic matrices, of type $N_1$, the branches are,
\begin{itemize}
 \item $2q^2$ branches of type $N_1$.
 \item $q^2(q-1)$ branches of type $N_3$.
 \item $2q^2$ branches of type $D_3$.
\end{itemize}
\end{prop}
\begin{proof} With respect to the form $\beta_3$ the common centralizer of a pair $(A,B)$ of commuting matrices of type $N_1$ is $Z_{Sp_4(\Fq)}(A,B) = \left\{\begin{pmatrix}a_0 & &b_0 & a_0d_0b_1\\  & d_0 & b_1 & b_2\\  &  & a_0 & \\ &  &  & d_0 \end{pmatrix}\mid a_0,d_0 = \pm 1 \right\}$. Let $C = \begin{pmatrix}a_0 & &b_0 & a_0d_0b_1\\  & d_0 & b_1 & b_2\\ &  & a_0 & \\ &  &  & d_0 \end{pmatrix}$, $X = \begin{pmatrix}x_0 & &y_0 & x_0z_0y_1\\  & z_0 & y_1 & y_2\\  &  & x_0 & \\ &  &  & z_0 \end{pmatrix}$, $C' = \begin{pmatrix}a'_0 & &b'_0 & a'_0d'_0b'_1\\  & d'_0 & b'_1 & b'_2\\ & & a'_0 & \\ &  &  & d'_0 \end{pmatrix}$. Then the equation $XC = C'X$ leads us to $a_0 = a'_0$, $d_0 = d'_0$, $b'_0 = b_0$, and $b_2' = b_2$. Which further gives the equation $z_0b_1 + a_0y_1 = x_0b'_1 + y_1d_0$. We look at cases $a_0 = d_0$, and $d_0 = -a_0$, to get the $b'_1$, and thus the branches as mentioned in the statement.
\end{proof}
\begin{prop}\label{PN123}
\begin{enumerate}
\item  For a pair of commuting matrices, of type $N_2$, there are $2q^3$ branches of the same type.
\item For a pair of commuting matrices of type $N_3$, there are $2q^3$ branches of the same type.
\end{enumerate}
\end{prop}
\begin{proof}
 The centralizer of a commuting pair of symplectic matrices $(A,B)$ of type $N_2$ is
 $\left\{ \begin{pmatrix}x_0 & y_1 &  & x_1 \\  & x_0 & &  \\  & z_2 & x_0 & -y_1 \\  & & & x_0\end{pmatrix}\mid x_0^2 =  1 \right\},$ with respect to the bilinear form $\beta_2$. 
 It is easy to see that it is a commutative group and hence $2q^3$ branches. Similarly, for a commuting tuple of type $N_3$, the centralizer is
$\left\{ \begin{pmatrix}x_0 &  & x_1 &y_1 \\  & x_0 &y_1 & y_2 \\ &  & x_0 &  \\  &  &  & x_0\end{pmatrix}\mid  x_0^2 = 1 \right\},$
with respect to the bilinear form $\beta_3$. Again, this one too is commutative.
\end{proof}
\noindent Now, in this process we do not get any new type. Thus, the process of computing branches stops here and we get the complete branching rules for $Sp_4(\Fq)$.

\subsection{Proof of Theorem~\ref{Main2}}\label{SecSp4}
The similarity class types (the existing ones, along with the three new types) are written in the order $\{A_1, A_2, A_3, A'_3, A_4, B_1, B_2, B_3, B_4, B_5, B_6, B_7,\\  B_8, B_9, C_1, C_2, C_3, C_4, D_1, D_2, D_3, N_1, N_2, N_3 \}$. The Propositions~\ref{PA1} to~\ref{PN123} give the branching rules for $Sp_4(\Fq)$. This is summarized in the form of a Table given in Table~\ref{tableIV}, and we call this matrix $BSp_4$. This can be used to compute $c_s(4,k,q)$. However, this gets quite complicated so we refrain from presenting that here.

\section{Branching and Commuting probability}\label{brcp}
In this section, we explore the relation between simultaneous conjugacy of commuting tuples (which is equivalent to computing branching matrix) and commuting probability. Let $G$ be a finite group and $l$ be the number of conjugacy classes in it. Let $g_1, \ldots, g_l$, denote the representatives of the conjugacy classes in $G$. We denote by $Z_G(g_i)$  the centralizer of $g_i$ in $G$, and $C_G(g_i)$ the conjugacy class of $g_i$. 
\begin{lemma}\label{lemmack}
Let $G$ be a finite group, and $k$ be a positive integer. Set $c_G(0)=1$. Then, the number of simultaneous conjugacy classes of $k$-tuples of commuting elements of $G$ is:
$$ c_G(k)= \sum_{i=1}^l c_{Z_G(g_i)}(k-1) = {\bf 1}. B_G^{k}. e_1$$
where ${\bf 1}$ is a row vector with all $1$'s, the vector $e_1$ is the column vector with $1$ at first place and $0$ elsewhere and $B_G$ is the branching matrix of $G$.
\end{lemma}
\begin{proof}
Since $c_G(0) = 1$, we get, $c_G(1) = \displaystyle\sum_{i=1}^l 1 = l$, the number of conjugacy classes in $G$. Now, given a $k$-tuple $(h_1, \ldots, h_k)$ of commuting elements in $G$ for $k \geq 2$, the $k-1$-tuple $(h_2, \ldots, h_k) \in Z_G(h_1)^{(k-1)}$. This gives a bijection between the orbits for the conjugation action of $Z_G(h_1)$ on $Z_G(h_1)^{(k-1)}$, and simultaneous conjugacy classes of commuting $k$-tuples in $G^{(k)}$, where the first component is conjugate to $h_1$, hence we get first equality. 

Now, we prove the second equality by induction on $k$. For the sake of clarity, we assume $B_G$ the branching matrix for a finite group $G$, is given with the rows and columns indexed by the conjugacy class types, with the first row and column indexed by the central type, i.e., the centre of $G$. So, by the definition of $G$, the first column consists of the branches of any central element of $G$, which is the number of conjugacy classes of $G$ of each type in $G$. When $k=1$, we have $c_G(1) = \displaystyle \sum_{i=1}^l 1 = \sum_{t}(B_{G})_{t1}$, where $t$ denotes the conjugacy class types in $G$. Recall, for the new types the entries $(B_{G})_{t1} = 0$, hence,  $c_G(1) = \bmf{1}.B_G.e_1$.

Now, assume induction upto $k$. Before going ahead, we make the following observation. From the branching matrix $B_G$, we can obtain the branching matrix for the centralizer subgroup $Z_G(t)$, for each of the types $t$. This we can do by taking a submatrix $B_{Z_G(t)}$ of those entries $(B_{G})_{ab}$, where conjugacy classes $C_a$ and $C_b$ occur in the list of branching of a class $C_t$ of type $t$. Notice that when $B_G$ is multiplied with itself, this submatrix multiplies only with itself. So, we have:
\begin{eqnarray*}
 c_G(k+1) &=& \sum_t (B_{G})_{t1}. c_{Z_G(t)}(k) = \sum_t (B_{G})_{t1}\left(\bmf{1}.B_{Z_G(t)}^k.e_t \right)~\text{ by induction}\\
 &= & \sum_t (B_{G})_{t1}\left(\bmf{1}.\left ((B_{G})_{cd}\right)^k.e_t\right)~\text{ types $c$, $d$ are branches of type $t$}\\
 &= & \sum_t (B_{G})_{t1}\left(\sum_a (B_G^k)_{at}\right)~\text{ type $a$ occurs as in the branching of type $t$}\\
 && \text{$(B_G)^k_{at} = 0$ if type $a$ is not a branch of type $t$.}\\
 &= &\sum_t\sum_a (B_G)^k_{at}.(B_{G})_{t1} =  \sum_a\left(\sum_t (B_G)^k_{at}.(B_{G})_{t1}\right)\\
 &= & \sum_a(B_G)^{k+1}_{a1} = \bmf{1}.B_G^{k+1}.e_1.
\end{eqnarray*}
This completes the proof.
\end{proof}

\begin{proof}[{\bf Proof of Theorem~\ref{TMain}}]
We prove this by induction on $k$. When $k = 2$, it is already known (see~\cite{ET} Theorem IV) that $cp_2(G) = \frac{l}{|G|}$, which is equal to $\frac{c_G(1)}{|G|}$. 
Thus we may assume $k \geq 3$. Using Lescot's formula (Equation~\ref{ELesc}), we get:
\begin{eqnarray*}
cp_{k+1}(G) &=& \frac{1}{|G|}\sum_{i=1}^l\frac{1}{|{C}_G(g_i)|^{k-1}} cp_k(Z_G(g_i))\\
&=& \frac{1}{|G|}\sum_{i=1}^l\frac{1}{|{C}_G(g_i)|^{k-1}}\frac{c_{Z_G(g_i)}(k-1)}{|Z_G(g_i)|^{k-1}}\text{~~~by induction}\\
&=& \frac{1}{|G|}\sum_{i=1}^l\frac{c_{Z_G(g_i)}(k-1)}{|G|^{k-1}} = \frac{1}{|G|^k}\sum_{i=1}^l c_{Z_G(g_i)}(k-1) = \frac{C_G(k)}{|G|^k}.
\end{eqnarray*}
The last step follows from Lemma~\ref{lemmack}. This completes the proof.
\end{proof}

\appendix
\section{Conjugacy classes in $Sp_2(\Fq)$ and $Sp_4(\Fq)$}\label{appendixa}
We first describe the conjugacy classes for $Sp_2(\Fq)\cong SL_2(\Fq)$ in the table below. We refer to the notation introduced in the Section~\ref{branchsp2}.
\vskip2mm
\begin{center}
\begin{tabular}{|c|c|c|c|c|} \hline 
\text{Representative} & \text{No. of Classes} & \text{Centralizer Size} & \text{Class Size} & \text{Notation} \\ \hline
$I_2$, $-I_2$ & $1$, $1$ & $q(q^2-1)$ & $1$ & $C$ \\ \hline
$\left(\begin{smallmatrix}1 & 1 \\ 0&1 \end{smallmatrix}\right)$, $\left(\begin{smallmatrix}-1 & 1 \\ 0&-1 \end{smallmatrix}\right)$ & $1$, $1$ & $2q$ & $\frac{q^2-1}{2}$ & $A_1$ \\ \hline
\begin{tabular}{c}$\left(\begin{smallmatrix} 1 & \eta \\ 0&1 \end{smallmatrix}\right)$, $\left( \begin{smallmatrix} -1 & \eta \\ 0 &-1 \end{smallmatrix}\right)$ \\ $\eta$ \text{\ is a non-square in\ } $\Fq$ \end{tabular} & $1$, $1$ & $2q$ & $\frac{q^2-1}{2}$ & $A_2$ \\ \hline
 $\left(\begin{smallmatrix}\lambda & 0 \\ 0& \lambda^{-1} \end{smallmatrix}\right),  \lambda \neq \pm 1 $ & $\frac{q-3}{2}$ & $(q-1)$ & $q(q+1)$ & $D$ \\ \hline
\begin{tabular}{c}$\left(\begin{smallmatrix}0 & 1\\ -1 & b \end{smallmatrix}\right)$ \\ $x^2-bx +1$ \text{\ is irreducible}\end{tabular} & $\frac{q-1}{2}$ & $q+1$ & $q(q-1)$ & $Ir$ \\ \hline
\end{tabular}
\end{center} \vskip2mm

Next we describe the conjugacy classes of $Sp_4(\Fq)$ which is already known (with respect to $\beta_1$) (see pages: 489-491 from Srinivasan~\cite{Sr} ). We have changed the nomenclature slightly, but kept it as close as possible to that of Srinivasan's. Also here we bring in $\kappa \in \mathbf{F}_{q^4}^*$, a generator of the multiplicative cyclic group. Define $\zeta = \kappa^{q^2-1} \in \mathbf{F}_{q^4} \setminus \Fqq$, $\theta = \kappa^{q^2+1} \in \Fqq^*$, which is, in fact, a generator of $\Fqq^*$. Also define $\eta = \theta^{q-1} \in \Fqq \setminus \Fq$, and $\gamma = \theta^{q+1}$, thus a generator of $\Fq$, and a non-square.\\
\def\arraystretch{1}
\begin{longtable}{|c|c|c|c|} \hline
Class Representative & $\begin{matrix} \text{Number of}\\ \text{Classes} \end{matrix}$ &$\begin{matrix} \text{Order of}\\ \text{Centralizer} \end{matrix}$ & $\begin{matrix} \text{Name of}\\ \text{Type} \end{matrix}$ \\ \hline
$\begin{matrix}a_0I_4\\ a_0 = \pm 1 \end{matrix}$ & 2 & $q^4(q^2-1)(q^4-1)$ & $A_1$ \\ \hline
$\begin{matrix} \left(\begin{smallmatrix}a_0&1&0 &0\\0&a_0&0&0\\0&0&a_0&0\\0&0&0&a_0\end{smallmatrix}\right),~\left(\begin{smallmatrix}a_0&\gamma&0 &0\\0&a_0&0&0\\0&0&a_0&0\\0&0&0&a_0\end{smallmatrix}\right)  \\ a_0= \pm1,~\langle \gamma \rangle = \F_q^* \end{matrix}$ & 4 & $2q^4(q^2-1)$ & $A_2$\\ \hline
$\begin{matrix} \left(\begin{smallmatrix}a_0&1&0 &0\\0&a_0&0&0\\0&0&a_0&-1 \\0&0&0&a_0\end{smallmatrix}\right) \\ a_0= \pm1 \end{matrix}$ & 2 & $2q^3(q-1)$ & $A_3$ \\ \hline
$\begin{matrix} \left(\begin{smallmatrix}a_0&1&0 &0\\0&a_0&0&0\\0&0&a_0&-\gamma\\0&0&0&a_0\end{smallmatrix}\right) \\ a_0= \pm1,~\langle \gamma \rangle = \F_q^*\end{matrix}$ & 2 & $2q^3(q-1)$ & $A'_3$ \\ \hline
$\begin{matrix} \left(\begin{smallmatrix}a_0&1&0 &0\\0&a_0&0&1\\-1&0&a_0&0\\0&0&0&a_0\end{smallmatrix}\right), \left(\begin{smallmatrix}a_0&\gamma&0 &0\\0&a_0&0&1\\-1&0&a_0&0\\0&0&0&a_0\end{smallmatrix}\right)\\ a_0= \pm1,~\langle \gamma \rangle = \F_q^*\end{matrix}$ & 4 & $2q^2$ & $A_4$ \\ \hline
$\begin{matrix} \left(\begin{smallmatrix}\zeta^i&0&0 &0\\0&\zeta^{-i}&0&0\\0&0&\zeta^{qi}&0\\0&0&0&\zeta^{-qi}\end{smallmatrix}\right), \\ i \in \{1,2,\ldots, (q^2-1)/4\}\end{matrix}$ & $\frac{1}{4}(q^2-1)$ & $q^2+1$ & $B_1$ \\ \hline
$\begin{matrix} \left(\begin{smallmatrix}\theta^i&0&0 &0\\0&\theta^{-i}&0&0\\0&0&\theta^{qi}&0\\0&0&0&\theta^{-qi}\end{smallmatrix}\right), \\ i: \theta^i,\theta^{-i},\theta^{qi}, \theta^{-qi}\\\text{ are distinct} \end{matrix}$ & $\frac{1}{4}(q-1)^2$ & $q^2-1$ & $B_2$ \\ \hline
$\begin{matrix} \left(\begin{smallmatrix}\gamma^i&0&0 &0\\0&\gamma^{-i}&0&0\\0&0&\gamma^{j}&0\\0&0&0&\gamma^{-j}\end{smallmatrix}\right), \\ i,j \in \{1,2,\ldots, (q-3)/2\} \\ i \neq j \end{matrix}$ & $\frac{1}{8}(q-3)(q-5)$ & $(q-1)^2$ & $B_3$ \\ \hline
$\begin{matrix} \left(\begin{smallmatrix}\eta^i&0&0 &0\\0&\eta^{-i}&0&0\\0&0&\eta^{j}&0\\0&0&0&\eta^{-j}\end{smallmatrix}\right), \\ i,j \in \{1,2,\ldots, (q-1)/2\} \\ i \neq j \end{matrix}$ & $\frac{1}{8}(q-1)(q-3)$ & $(q+1)^2$ & $B_4$ \\ \hline
$\begin{matrix} \left(\begin{smallmatrix}\eta^i&0&0 &0\\0&\eta^{-i}&0&0\\0&0&\gamma^{j}&0\\0&0&0&\gamma^{-j}\end{smallmatrix}\right), \\ i \in \{1,2,\ldots, (q-1)/2\} \\ j \in \{1,2,\ldots, (q-3)/2\} \end{matrix}$ & $\frac{1}{4}(q-1)(q-3)$ & $(q^2-1)$ & $B_5$ \\ \hline
$\begin{matrix} \left(\begin{smallmatrix}\eta^i&0&0 &0\\0&\eta^{-i}&0&0\\0&0&\eta^{i}&0\\0&0&0&\eta^{-i}\end{smallmatrix}\right), \\ i \in \{1,2,\ldots, (q-1)/2\}  \end{matrix}$ & $\frac{1}{2}(q-1)$ & $q(q+1)(q^2-1)$ & $B_6$ \\ \hline
$\begin{matrix} \left(\begin{smallmatrix}\eta^i&0&1 &0\\0&\eta^{-i}&0&1\\0&0&\eta^{i}&0\\0&0&0&\eta^{-i}\end{smallmatrix}\right), \\ i \in \{1,2,\ldots, (q-1)/2\}  \end{matrix}$ & $\frac{1}{2}(q-1)$ & $q(q+1)$ & $B_7$ \\ \hline
$\begin{matrix} \left(\begin{smallmatrix}\gamma^i&0&0 &0\\0&\gamma^{-i}&0&0\\0&0&\gamma^{i}&0\\0&0&0&\gamma^{-i}\end{smallmatrix}\right), \\ i \in \{1,2,\ldots, (q-3)/2\} \end{matrix}$ & $\frac{1}{2}(q-3)$ & $q(q-1)(q^2-1)$ & $B_8$ \\ \hline
$\begin{matrix} \left(\begin{smallmatrix}\gamma^i&0&1 &0\\0&\gamma^{-i}&0&1\\0&0&\gamma^{i}&0\\0&0&0&\gamma^{-i}\end{smallmatrix}\right), \\ i \in \{1,2,\ldots, (q-3)/2\} \end{matrix}$ & $\frac{1}{2}(q-3)$ & $q(q-1)$ & $B_9$ \\ \hline
$\begin{matrix} \left(\begin{smallmatrix}\eta^i&0&0 &0\\0&\eta^{-i}&0&\\0&0&a_0&0\\0&0&0&a_0\end{smallmatrix}\right), \\ i \in \{1,2,\ldots, (q-1)/2\}\\a_0 = \pm1  \end{matrix}$ & $(q-1)$ & $q(q+1)(q^2-1)$ & $C_1$ \\ \hline
$\begin{matrix} \left(\begin{smallmatrix}\eta^i&0&0 &0\\0&\eta^{-i}&0&0\\0&0&a_0&1\\0&0&0&a_0\end{smallmatrix}\right),~\left(\begin{smallmatrix}\eta^i&0&0 &0\\0&\eta^{-i} &0&0\\0&0&a_0&\gamma\\0&0&0&a_0\end{smallmatrix}\right)  \\ i \in \{1,2,\ldots, (q-1)/2\}\\a_0 = \pm1 \end{matrix}$ & $2(q-1)$ & $2q(q+1)$ & $C_2$\\ \hline
$\begin{matrix} \left(\begin{smallmatrix}\gamma^i&0&0 &0\\0&\gamma^{-i}&0&\\0&0&a_0&0\\0&0&0&a_0\end{smallmatrix}\right), \\ i \in \{1,2,\ldots, (q-3)/2\}\\a_0 = \pm1  \end{matrix}$ & $(q-3)$ & $q(q-1)(q^2-1)$ & $C_3$ \\ \hline
$\begin{matrix} \left(\begin{smallmatrix}\gamma^i&0&0 &0\\0&\gamma^{-i}&0&0\\0&0&a_0&1\\0&0&0&a_0\end{smallmatrix}\right),~\left(\begin{smallmatrix}\gamma^i&0&0 &0\\0&\gamma^{-i} &0&0\\0&0&a_0&\gamma\\0&0&0&a_0\end{smallmatrix}\right)  \\ i \in \{1,2,\ldots, (q-3)/2\}\\a_0 = \pm1 \end{matrix}$ & $2(q-3)$ & $2q(q-1)$ & $C_4$\\ \hline
$\begin{matrix} \left(\begin{smallmatrix}1&0&0 &0\\0&1&0&\\0&0&-1&0\\0&0&0&-1\end{smallmatrix}\right) \end{matrix}$ & $1$ & $q^2(q^2-1)^2$ & $D_1$ \\ \hline
$\begin{matrix} \left(\begin{smallmatrix}a_0&0&0 &0\\0&a_0&0&\\0&0&-a_0&1\\0&0&0&-a_0\end{smallmatrix}\right),~\left(\begin{smallmatrix}a_0&0&0 &0\\0&a_0&0&\\0&0&-a_0&\gamma\\0&0&0&-a_0\end{smallmatrix}\right) \\ a_0 = \pm 1\end{matrix}$ & $4$ & $2q^2(q^2-1)$ & $D_2$ \\ \hline
$\begin{matrix} \left(\begin{smallmatrix}1&1&0 &0\\0&1&0&\\0&0&-1&1\\0&0&0&-1\end{smallmatrix}\right),~\left(\begin{smallmatrix}1&1&0 &0\\0&1&0&\\0&0&-1&\gamma\\0&0&0&-1\end{smallmatrix}\right) \\ \left(\begin{smallmatrix}1&\gamma&0 &0\\0&1&0&\\0&0&-1&1\\0&0&0&-1\end{smallmatrix}\right),~\left(\begin{smallmatrix}1&\gamma&0 &0\\0&1&0&\\0&0&-1&\gamma\\0&0&0&-1\end{smallmatrix}\right)\end{matrix}$ & $4$ & $4q^2$ & $D_3$ \\ \hline
\caption{Conjugacy Classes of $Sp_4(\Fq)$}\label{TSp4Con}
\end{longtable}
\section{Conjugacy classes of $O_2^+(\Fq)$, and $O_2^-(\Fq)$}\label{appendixb}
For the branching in group $Sp_4(\Fq)$ we require the conjugacy classes in $2\times 2$ orthogonal group over $\Fq$. There are two such groups one split, denoted as $O_2^+(\Fq)$, and other non-split, denoted as $O_2^-(\Fq)$. The split orthogonal group $O_2^+(\Fq) =\{A \in GL_2(\Fq) \mid A.\beta.{}^t\!A = \beta\}$ is defined with respect the symmetric bilinear form $\beta = \begin{pmatrix}&1\\1&\end{pmatrix}$. If we take the symmetric bilinear form to be $ \begin{pmatrix}1 & \\ & -1\end{pmatrix}$, it is equivalent to $\beta$, and hence gives orthogonal group, conjugate to $O_2^+(\Fq)$. The conjugacy classes of this group $O_2^+(\Fq)$ are:
\begin{center}
 \begin{tabular}{|c|c|c|c|} \hline
 \text{Representative} & \text{No. of Classes} & \text{Centralizer Size} & \text{Class Size}\\ \hline
$I_2$ , $-I_2$ & $1,1$ & $2(q-1)$ & $1$  \\ \hline 
$\begin{pmatrix}  \gamma^i & 0 \\ 0 & \gamma^{-i} \end{pmatrix}$ & $\displaystyle\frac{q-3}{2}$ & $q-1$ & $2$ \\ 
$i = 1, \ldots, (q-3)/2$ & & & \\ \hline
$\begin{pmatrix} 0 & 1 \\ 1 & 0\end{pmatrix}$ ,$\begin{pmatrix} 0 & \gamma \\ \gamma^{-1} & 0\end{pmatrix}$ & $1,1$ & $4$ & $\displaystyle\frac{q-1}{2}$. \\ \hline
 \end{tabular}
\end{center}\vskip2mm 
The non-split orthogonal group $O_2^-(\Fq) = \{A \in GL_2(\Fq)\mid A.\beta{}.^t\!A = \beta\}$ is defined with respect to the symmetric bilinear form $\beta=\begin{pmatrix}1 & \\ & -\gamma\end{pmatrix}$, where $\gamma$ is a non-square such that $\langle \gamma \rangle = \Fq^*$. The conjugacy classes in $O_2^-(\Fq)$ are:
\begin{center}
 \begin{tabular}{|c|c|c|c|}\hline
  \text{Representative} & \text{No. of Classes} & \text{Centralizer Size} & \text{Class Size}\\ \hline
  $I_2$, $-I_2$ & $1,1$ & $2(q+1)$ & 1\\ \hline
  $\begin{pmatrix} a_0 & a_1\\ \gamma a_1 & a_0\end{pmatrix}$ &$\displaystyle\frac{q-1}{2}$ & $q + 1$ & $2$\\ 
  $a_0^2 + \gamma a_1^2 = 1$, $a_0 \neq 0$ & & & \\ \hline
  $\begin{pmatrix}1 & \\ & -1\end{pmatrix}$, $\begin{pmatrix}-1 & \\ & 1\end{pmatrix}$ & $1,1$ & 4 & $\displaystyle\frac{q+1}{2}$. \\ \hline
 \end{tabular}
\end{center}

\bibliographystyle{alpha}

\end{document}